\documentclass[reqno,a4paper,12pt]{amsart} 

\usepackage{amsmath,amscd,amsfonts,amssymb}
\usepackage{mathrsfs,dsfont}

\usepackage{caption}
\usepackage{subcaption}

\usepackage{tikz}
 \usetikzlibrary{calc}

\definecolor{myblue}{RGB}{80,80,160}
\definecolor{mygreen}{RGB}{80,160,80}
\definecolor{myred}{RGB}{160,20,20}
\definecolor{mygrey}{RGB}{100,100,100}

\tikzstyle{mystyle}=[line width=0.3mm, myblue] 
\tikzstyle{vertex}  = [draw,circle,fill] 
\tikzstyle{xnode}  = [draw=black, fill=mygreen]
\tikzstyle{xnodelabel}  = [black, below, yshift=-0.1cm]
\tikzstyle{ygnode}  = [draw=black, fill=myblue]
\tikzstyle{ygnodelabel}  = [black, above, yshift=0.1cm]
\tikzstyle{unode}  = [draw=black, fill=myblue]
\tikzstyle{vnode}  = [draw=black, fill=mygreen]
\tikzstyle{unodelabel}  = [black, left, xshift=-0.25cm, yshift=-0.05cm]
\tikzstyle{vnodelabel}  = [black, right, xshift=0.25cm, yshift=-0.05cm]
\tikzstyle{myedge}  = [color=black, semithick]
\tikzstyle{mylightedge}  = [color=mygrey, thin]
\tikzstyle{edgelabel}  = [font=\scriptsize,color=myred]
\tikzstyle{uedgelabel}  = [font=\scriptsize,color=myred,above,sloped,pos=0.65]
\tikzstyle{vedgelabel}  = [font=\scriptsize,color=myred,above,sloped,pos=0.35]
\tikzstyle{specialnode}=[line width=0.65mm, myred] 
\tikzstyle{myroundrect}=[semithick,color=mygrey,rounded corners=10]

\def\mycirc{12pt}
\def\mybigcirc{24pt}
\def\mylittlecirc{6pt}

\def\myk{p}

\numberwithin{equation}{section}
\numberwithin{figure}{section}

\newcommand\R{\mathbb{R}}
\newcommand\C{\mathbb{C}}

\newcommand\Z{\mathbb{Z}}

\newcommand\Om{\Omega}

\newcommand\TT{\mathcal{T}}
\renewcommand\S{\mathcal{S}}
\newcommand\Snm{\mathcal{S}_{n,m}}

\renewcommand\leq{\leqslant}
\renewcommand\geq{\geqslant}
\newcommand\sbt{\subset}

\newcommand{\supp}{\operatorname{supp}}

\theoremstyle{plain}
\newtheorem{thm}{Theorem}[section]
\newtheorem{lem}[thm]{Lemma}

\newtheorem{prop}[thm]{Proposition}

\newtheorem*{claim*}{Claim}

\newcommand{\thmref}[1]{Theorem~\ref{#1}}
\newcommand{\secref}[1]{Section~\ref{#1}}

\newcommand{\lemref}[1]{Lemma~\ref{#1}}
\newcommand{\defref}[1]{Definition~\ref{#1}}
\newcommand{\propref}[1]{Proposition~\ref{#1}}

\newcommand{\figref}[1]{Figure~\ref{#1}}

\theoremstyle{definition}
\newtheorem{definition}[thm]{Definition}
\newtheorem*{definition*}{Definition}
\newtheorem*{remarks*}{Remarks}
\newtheorem*{remark*}{Remark}

\newenvironment{enumerate-roman}
{\begin{enumerate}
\addtolength{\itemsep}{5pt}
}
{\end{enumerate}}

\newenvironment{enumerate-alph}
{\begin{enumerate}
\addtolength{\itemsep}{5pt}
}
{\end{enumerate}}

\newenvironment{enumerate-num}
{\begin{enumerate}
\addtolength{\itemsep}{5pt}
}
{\end{enumerate}}

\newenvironment{enumerate-text}
{\begin{enumerate}
\addtolength{\itemsep}{5pt}
}
{\end{enumerate}}

\begin{document}

\title{Support of extremal doubly stochastic arrays}

\author{Mark Mordechai Etkind}
\address{Department of Mathematics, Bar-Ilan University, Ramat-Gan 5290002, Israel}
\email{mark.etkind@mail.huji.ac.il}

\author{Nir Lev}
\address{Department of Mathematics, Bar-Ilan University, Ramat-Gan 5290002, Israel}
\email{levnir@math.biu.ac.il}

\date{January 12, 2025}
\subjclass[2020]{05B20, 05C05, 15B51}
\keywords{Doubly stochastic arrays, extremal matrices, Birkhoff's theorem}
\thanks{Research supported by ISF Grant No.\ 1044/21 and ERC Starting Grant No.\ 713927.}

\begin{abstract}
An $n \times m$ array with nonnegative entries is called \emph{doubly stochastic} if the sum of its entries at each row is $m$ and at each column is $n$. The set of all $n \times m$ doubly stochastic arrays is a convex polytope with finitely many extremal points. The main result of this paper characterizes the possible sizes of the supports of all extremal $n \times m$ doubly stochastic arrays. In particular we prove that the minimal size of the support of an $n \times m$ doubly stochastic array is $n + m - \gcd(n,m)$. Moreover, for $m=kn+1$ we also characterize the structure of the support of the extremal arrays.
\end{abstract}

\maketitle


\section{Introduction}

\subsection{} 
An $n \times m$ matrix $A = (a_{ij})$ is called
a \emph{doubly stochastic array} (with uniform marginals) if
its entries are nonnegative,  $a_{ij} \geq 0$, and 
\begin{equation}
\label{eqDS1.1}
\sum_{j=1}^{m} a_{ij} = m, \quad 1 \leq i \leq n,
\end{equation}
\begin{equation}
\label{eqDS1.2}
\sum_{i=1}^{n} a_{ij} = n, \quad 1 \leq j \leq m,
\end{equation}
that is, the sum of the entries at each row is $m$ and at each column is $n$.

For example (see \cite[p.\ 65]{Car96}), the two matrices
\begin{equation}
\label{eqDS1.3}
\begin{bmatrix}
3 & 0 & 0 & 1\\
0 & 3 & 0 & 1\\
0 & 0 & 3 & 1
\end{bmatrix}
\quad \text{and} \quad 
\begin{bmatrix}
1 & 0 & 0 & 3\\
2 & 2 & 0 & 0\\
0 & 1 & 3 & 0
\end{bmatrix}
\end{equation}
are  $3 \times 4$ doubly stochastic arrays.

We will use $\Snm$ to denote the set of all
 $n \times m$  doubly stochastic arrays.
This set is  a convex polytope 
(of dimension $(n-1)(m-1)$) in the linear space of all
$n \times m$ matrices. Hence $\Snm$ coincides
with  the convex hull of its finitely many extremal points. We
recall that an element $A \in \Snm$
is  \emph{extremal}
if it cannot be represented as a convex
combination of two other elements
of   $\Snm$ that are
both different from $A$.

For \emph{square} $n \times n$ matrices,
a  classical theorem due to Birkhoff \cite{Bir46}
characterizes the extremal 
doubly stochastic arrays
as follows: \emph{an array $A \in \S_{n,n}$ is 
extremal if and only if $\frac1{n} A$ is a permutation
matrix}.

Let us say that two $n \times m$ matrices
are \emph{equivalent} if one can be transformed into the other
by a permutation of rows and columns.
In this language, Birkhoff's theorem states that
up to equivalence there is only one extremal
 $n \times n$  doubly stochastic array,
namely, the identity matrix times the scalar $n$.

\emph{Nonsquare} $n \times m$ extremal  doubly stochastic arrays have
been studied by different authors, see 
\cite{Car96}, \cite{Li96} and the references therein.
This case is more complicated and in general 
there is more than one equivalence class of extremal
arrays. For example,
the two matrices in \eqref{eqDS1.3}
are both extremal but not equivalent.

\subsection{}
In \cite[Question 7]{KP22} the authors posed the following question,
which arose in connection with a tiling problem
in finite abelian groups: what is
the smallest possible size of the support of a
  doubly stochastic  $n \times m$ array?

We recall that 
 the \emph{support} of an $n \times m$
matrix $A = (a_{ij})$ is the set
\begin{equation}
\supp A = \{(i,j) : a_{ij} \neq 0\}.
\end{equation}

In this paper we  obtain a complete answer to the
 question above:

\begin{thm}
\label{thmA6.7}
For all $n,m$ the minimal size of the
support of an array $A \in \Snm$ is
equal to $n + m - \gcd(n,m)$.
\end{thm}

This result 
was also obtained 
by M.\ Loukaki in  \cite{Lou23} and 
we acknowledge her priority. 
In \cite[Section 4]{KP22}
the authors established that
the result holds
in the special case where either
$m = kn$ or $m = kn+1$.

\subsection{}
It is known, see \cite[Corollary 2]{Li96}, that an array
$A \in \Snm$ is extremal if and only if  its support
is minimal with respect to inclusion, that is,
there is no $B \in \Snm$ such that $\supp B$
is a proper subset of $\supp A$.
This implies that any array
which minimizes the
size of the support must also be an extremal array.

For example, each one of
the two $3 \times 4$  matrices
in \eqref{eqDS1.3}
has $6 = 3+4-1$ nonzero entries,
which is 
the minimal size of the support of
a doubly stochastic $3 \times 4$ array. 
It follows that these two matrices
are extremal.

One can thus pose a more general version
of the question above, namely, what are
the possible sizes of the support
of an extremal $n \times m$ 
 doubly stochastic  array?

In the special case
where  one of $n,m$ divides the other,
say,  if $m=kn$, it has been known
that
there is only one equivalence class of 
extremal $n \times kn$ doubly stochastic arrays,
and the support
size of every extremal array 
is $kn$,  see \cite[Proposition 4]{Car96}.

Our next result provides a complete answer 
to the question
in the more complicated  case where
 neither one of $n,m$ divides the other.
That is, the result characterizes 
the possible sizes of the
supports   of all extremal $n \times m$ arrays.

\begin{thm}
\label{thmA5.1}
Assume that neither one of $n,m$ divides the other.
Then any extremal array  $A \in \Snm$
has support size of   the form
\begin{equation}
\label{eqA5.1.1}
|\supp A| = n + m - s
\end{equation}
where 
$1 \leq s \leq \gcd(n,m)$. 
Conversely,  for any integer 
$s$, 
$1 \leq s \leq \gcd(n,m)$,
 there exists an  extremal array 
 $A \in \Snm$  whose support size 
is given by \eqref{eqA5.1.1}.
\end{thm}

For certain special values of $n$ and $m$, 
examples of extremal $n \times m$ arrays whose support is
not of minimal size were also constructed in \cite[Theorem~3]{Lou23}.

\subsection{}
It follows from \thmref{thmA5.1} that if
$n,m$ are \emph{coprime}, then all the
extremal $n \times m$ arrays
are of the same support size, 
namely, $n+m-1$.
Hence an array  is extremal
if and only if it minimizes the support size
among all the  $n \times m$
  doubly stochastic  arrays.
So in this case  one can characterize
extremality by the size of the support:

\begin{thm}
\label{thmA5.2}
Let $n,m$ be coprime. Then 
an array  $A \in \Snm$  is 
extremal if and only if
\begin{equation}
\label{eqA5.2.1}
|\supp A| = n + m - 1.
\end{equation}
\end{thm}

This result 
was obtained also in  \cite{Lou23} 
using a different approach.
In the special case where 
$m = kn+1$ the result was proved
in \cite[Proposition 6]{Car96}.

If one of $n,m$ divides the other,
say, $m=kn$, then there is only one equivalence class of 
extremal   arrays 
and the support size of every extremal array
is equal to  $kn$. Hence  in this case
an array $A \in \S_{n,kn}$  is 
extremal if and only if
$|\supp A| = kn$.

However, for  all the remaining values of $n,m$,
that is, if $n,m$ are neither coprime nor
any one of them divides the other,
there is no characterization of
extremality by the size of the support.
One can show that the  only sufficient condition
for the extremality of   $A \in \Snm$
is $|\supp A| = n + m - \gcd(n,m)$,
but this condition is not necessary.

\subsection{}
A more general problem is to determine not only
the possible size, but also a structural
characterization, of the support of an
extremal $n \times m$ doubly stochastic array.
That is, to obtain an effective necessary and sufficient
condition for a subset
$\Om \sbt \{1, \dots, n\} \times \{1, \dots, m\}$
to be  the support of  some  
extremal $n \times m$   array.

We note that if an extremal $n \times m$ array with a given support
does exist, then it is unique, see \cite[Theorem 1]{Li96}.

In this paper we solve this problem
in the special case where $m = kn+1$.

\begin{thm}
\label{thmF1.1}
Let  $m = k n+1$. Then a subset
 $\Om \sbt \{1, \dots, n\} \times \{1, \dots, m\}$
is the support of an extremal
array $A \in \Snm$ if and only if the following two
conditions hold:
\begin{enumerate-roman}
\item
\label{itF1.2.1}
For each $1 \leq i \leq n$, the set 
$\{j : (i,j) \in \Om\}$ contains 
exactly $k+1$ elements;
\item
\label{itF1.2.2}
 $\Om$ does not contain any ``cycle'', that is,
any sequence of the form
\begin{equation}
\label{eqF1.2}
(i_1, j_1), \; (i_2, j_1), \; 
(i_2, j_2), \; (i_3, j_2), \; 
\dots, \;
(i_s, j_s), \; (i_1, j_s)
\qquad (s \geq 2)
\end{equation}
where $i_1, \dots, i_s$ are distinct row indices
and $j_1, \dots, j_s$  are distinct column indices.
\end{enumerate-roman}
\end{thm}

The necessity of the  conditions
\ref{itF1.2.1}, \ref{itF1.2.2}
 is  known, see
\cite[Propositions 2, 6]{Car96}.
The novel part of the result is  that 
these two conditions are also sufficient.

For example, it is easy to check that the subset
\begin{equation}
\label{eqDF1.3}
\Om = 
\begin{bmatrix}
\bullet & 0 & 0 & \bullet\\
\bullet & \bullet & 0 & 0\\
0 & \bullet & \bullet & 0
\end{bmatrix}
\end{equation}
of the $3 \times 4$ array,
whose elements are indicated by `$\bullet$',
satisfies the conditions 
\ref{itF1.2.1} and \ref{itF1.2.2}
in \thmref{thmF1.1} (with $n=3$ and $k=1$).
Hence $\Om$ must be the support of some
extremal array $A \in \S_{3,4}$.
Indeed, $A$ is  the matrix 
on the right hand side of \eqref{eqDS1.3}.

As an application of \thmref{thmF1.1} we
find the total number of extremal  
$n \times (n+1)$ doubly stochastic arrays:

\begin{thm}
\label{thmE12.1}
The total number of extremal arrays in $\S_{n,n+1}$
is exactly  $n!  (n+1)^{n-1}$. 
\end{thm}

\subsubsection*{Note added in proof}
We later learned that \thmref{thmE12.1} is a special case of a more
general result \cite[Theorems 14 and 15]{KW67}
which gives the total number of 
extremal  $n \times (kn \pm 1)$ doubly stochastic arrays.

\subsection{}
Loukaki posed in \cite[Remark 2]{Lou23} the following question: 
Do there exist two extremal 
$n \times m$ doubly stochastic arrays $A,B$
whose supports are of
the least possible  size $n+m - \gcd(n,m)$,  and
such that $A,B$ have the same set of entries (counted
with multiplicities) but $A$ and $B$ are
not equivalent?

We will prove that the answer is affirmative
for $n \times (n+1)$ arrays:

\begin{thm}
\label{thmE11.3}
For any $n \geq 6$
there exist two extremal arrays
 $A,B \in \S_{n,n+1}$ that
have the same set of entries counted
with multiplicities, but $A$ and $B$ are
not equivalent.
\end{thm}

In fact  the two arrays $A,B$
in this result are not only
extremal but moreover 
have  supports of the least 
possible size, since 
 $n$ and $n+1$ are coprime
(\thmref{thmA5.2}).

We note that the result 
in \thmref{thmE11.3}  does not hold
for $n \leq 5$.

\subsection{}
The rest of the paper  is organized as follows. 
In \secref{secP1} we give  some preliminary background.
In \secref{secD1} we prove  that
the minimal size of the
support of an
 $n \times m$ doubly stochastic
array is
 $n + m - \gcd(n,m)$
(\thmref{thmA6.7}).
In particular we present methods
for constructing doubly stochastic  
arrays of minimal
support size.

In \secref{secD2} we  
characterize the possible sizes of the supports
of all the extremal $n \times m$ doubly stochastic  arrays
(\thmref{thmA5.1}).
In \secref{secD4}
we characterize the support structure of the
extremal $n \times (kn+1)$ arrays  (\thmref{thmF1.1}).
We then use this result in order to
determine the total number of extremal $n \times (n+1)$
  arrays (\thmref{thmE12.1}).

In \secref{secD3} 
we construct examples
of non-equivalent extremal arrays with the same 
set of entries
counted with multiplicities 
 (\thmref{thmE11.3}).

In the last \secref{secR1} we 
 give some motivational remarks, 
relating the notion of
doubly stochastic arrays
to some other mathematical topics.


\section{Preliminaries}
\label{secP1}

In this section we provide 
  some preliminary background
on  doubly stochastic arrays.
First we discuss the notion of the 
graph associated to an $n \times m$ matrix,
 and then state several basic results on doubly stochastic arrays
that will be used later on.

\subsection{The associated graph of a matrix}
To every $n \times m$ complex-valued
matrix  $A = (a_{ij})$  one can associate a 
weighted (undirected)
bipartite graph $G(A) = (U,V,E,w)$
 defined as follows.
The set of vertices of the graph is the union of two disjoint sets
$U = \{u_1, \dots, u_n\}$ and
$V = \{v_1, \dots, v_m\}$ (corresponding to rows
and columns of $A$ respectively)
which form a bipartition of the graph.
The set of edges $E$ includes the edge 
 $\{u_i, v_j\}$ joining $u_i$ and $v_j$
if and only if the matrix entry $a_{ij}$ is nonzero.
Finally, the graph is endowed with 
a weight function $w: E \to \C$ which assigns the 
(nonzero) weight $a_{ij}$ to the edge $\{u_i, v_j\}$.

For example, \figref{fig:graphexamples} shows
the graphs associated to the
 two matrices in \eqref{eqDS1.3}.

If $x$ is any one of the  vertices of the graph $G(A)$,
then the  sum of the weights of all the edges incident to $x$
will be called the \emph{weight of the vertex $x$}.

We observe that an $n \times m$ matrix $A$ is doubly 
stochastic if and only if the graph $G(A)$ has
the  following three
properties: (i) the edges of the graph
have all positive weights;
(ii)  every vertex from $U$ has weight  $m$;
and (iii) every vertex from $V$ has weight  $n$.


\begin{figure}[ht]
\centering

\begin{subfigure}[t]{0.475\textwidth}
        \centering

\begin{tikzpicture}[scale=0.32, style=mystyle]

\draw[myroundrect] (-7,-4) rectangle  (15,16);

\coordinate (v4) at (8,0);
\coordinate (u3) at (0,2);
\coordinate (v3) at (8,4);
\coordinate (u2) at (0,6);
\coordinate (v2) at (8,8);
\coordinate (u1) at (0,10);
\coordinate (v1) at (8,12);

\draw[myedge] (u1) edge node[edgelabel,above,pos=0.65]{$3$}  (v1);
\draw[myedge] (u2) edge node[edgelabel,above,pos=0.65]{$3$}  (v2);
\draw[myedge] (u3) edge node[edgelabel,above,pos=0.75]{$3$}  (v3);
\draw[myedge] (v4) edge node[edgelabel,right,pos=0.85]{$1$}  (u1);
\draw[myedge] (v4) edge node[edgelabel,above,pos=0.65]{$1$}  (u2);
\draw[myedge] (v4) edge node[edgelabel,below]{$1$}  (u3);

\foreach \j in {1,...,4} { \draw[vnode] (v\j) node[vnodelabel]{$v_\j$} circle (\mycirc); }
\foreach \j in {1,...,3} { \draw[unode] (u\j) node[unodelabel]{$u_\j$} circle (\mycirc); }

\end{tikzpicture}
    \end{subfigure}
~
\begin{subfigure}[t]{0.475\textwidth}
        \centering
\begin{tikzpicture}[scale=0.32, style=mystyle]

\draw[myroundrect] (-7,-4) rectangle  (15,16);

\coordinate (v4) at (8,0);
\coordinate (u3) at (0,2);
\coordinate (v3) at (8,4);
\coordinate (u2) at (0,6);
\coordinate (v2) at (8,8);
\coordinate (u1) at (0,10);
\coordinate (v1) at (8,12);

\draw[myedge] (u1) edge node[edgelabel,above,pos=0.5]{$1$}  (v1);
\draw[myedge] (u2) edge node[edgelabel,above,pos=0.65]{$2$}  (v2);
\draw[myedge] (u3) edge node[edgelabel,below,pos=0.35]{$3$}  (v3);
\draw[myedge] (v4) edge node[edgelabel,left,pos=0.15]{$3$}  (u1);
\draw[myedge] (v1) edge node[edgelabel,above,pos=0.55]{$2$}  (u2);
\draw[myedge] (v2) edge node[edgelabel,above,pos=0.75]{$1$}  (u3);

\foreach \j in {1,...,4} { \draw[vnode] (v\j) node[vnodelabel]{$v_\j$} circle (\mycirc); }
\foreach \j in {1,...,3} { \draw[unode] (u\j) node[unodelabel]{$u_\j$} circle (\mycirc); }

\end{tikzpicture}
    \end{subfigure}

\caption{The graphs associated to the two matrices in \eqref{eqDS1.3}.}
\label{fig:graphexamples}
\end{figure}


\subsection{Basic results on extremal arrays}

We collect some basic 
results on 
extremal  $n \times m$ doubly 
stochastic  arrays
 that will be used in the paper.

\begin{prop}[{see \cite[Theorem 1]{Li96}}]
\label{propPR5.71}
An array  $A \in \Snm$ is extremal if and
only if the two conditions 
$B \in \Snm$, $\supp B = \supp A$
imply that 
$B=A$.
\end{prop}

\begin{prop}[{see \cite[Corollary 2]{Li96}}]
\label{propPR5.8}
An array  $A \in \Snm$ is extremal if and
only if 
there is no array $B \in \Snm$ such that $\supp B$
is a proper subset of $\supp A$.
\end{prop}

\begin{prop}[{see \cite[Proposition 2]{Car96}}]
\label{propPR5.4}
An array  $A \in \Snm$ is extremal if and
only if the graph $G(A)$ contains no cycles.
\end{prop}

\begin{prop}
\label{propPR5.2}
Let $A \in \Snm$ be an extremal array.
Then $A$ is an integer-valued matrix. Moreover,
 all the entries of $A$ are
integral multiples of $\gcd(n,m)$.
\end{prop}

\begin{proof}
Let $d = \gcd(n,m)$. Since $A$ is doubly stochastic,
then for every vertex $x$ of the graph $G(A)$ the
sum of the weights of all the edges incident to $x$
is in $\Z d$  (since this sum is either $n$ or $m$).
Hence if a vertex $x$ is incident to some edge $e$
of weight not in $\Z d$, then $x$
must be incident to at least one other edge $e'$
also of weight not in $\Z d$.
This implies that if $H$ is the subgraph of $G(A)$ induced by
all the edges with weights not in 
$\Z d$, then 
every vertex of $H$ must have degree at least two.
It follows that if $H$ is nonempty, then $H$ must contain
a cycle, which is not possible due to the
extremality of the matrix $A$ (\propref{propPR5.4}).
Thus $H$ has to be empty,
that is, the weight of
every edge of $G(A)$ is in $\Z d$, and
so all the
entries of $A$ are
integral multiple of $d$.
\end{proof}

\begin{prop}[{see \cite[Proposition 4]{Car96}}]
\label{propPR5.7}
If $m=kn$, then any extremal
  $n \times m$ doubly stochastic array 
is equivalent to the block diagonal matrix 
\setcounter{MaxMatrixCols}{50}
\begin{equation}
\label{eqPR5.7.1}
\begin{bmatrix}
n & n & \cdots  & n   \\
& & & & n & n & \cdots  & n \\
& & & &  &  &   &  & \ddots \\
& & & &  &  &   &  & &  n & n & \cdots  & n 
\end{bmatrix}
\end{equation}
where in each row there are exactly $k$ nonzero entries.
\end{prop}

Notice that Birkhoff's theorem
is the case $k=1$ in \propref{propPR5.7}.

The following proof is different from the one in \cite{Car96}.

\begin{proof}[Proof of \propref{propPR5.7}]
By \propref{propPR5.2},
 all the entries of $A$ are
integral multiples of $n$.
Since the entries are nonnegative and
the sum at each column  is
$n$, there can be no entry greater than $n$.
Hence  every nonzero entry must be
equal to $n$. Since
the sum of the entries at each row is $kn$,
then in each row there are exactly $k$ nonzero entries.
Similarly,  the sum  of the entries  at each column  is $n$ 
and hence in each column there is exactly one nonzero entry.
Thus $A$ must be
 equivalent to the matrix in \eqref{eqPR5.7.1}.
\end{proof}


\section{Arrays with minimal support size}
\label{secD1}

In this section  our
 main goal is to prove \thmref{thmA6.7},
which states that the minimal size of the
support of an
 $n \times m$ doubly stochastic
array is
 $n + m - \gcd(n,m)$.

To prove this result we will first show that
the support  size
of any $n \times m$ doubly stochastic
array is bounded from below by
the  value $n + m - \gcd(n,m)$.

 Then we will turn to construct
examples of arrays that attain the
 minimal support size. First, we will 
show that the problem can be reduced to
the special case when $n,m$
are \emph{coprime}.
Then, we will present two different methods for 
constructing
$n \times m$   doubly stochastic  
arrays of minimal
support size.

\subsection{Lower bound for the support  size
of a doubly stochastic array}
We begin by showing that an $n \times m$ doubly stochastic 
array cannot have support of 
size smaller than
$ n + m - \gcd(n,m)$, which
establishes the lower bound 
in \thmref{thmA6.7}. In fact we will prove a 
 stronger version of this result,
 which establishes that the same
lower bound holds also for complex-valued 
matrices, that is, even without assuming
that the matrix entries are nonnegative.

\begin{thm}
\label{thmD1.1}
Let $A = (a_{ij})$ be an  $n \times m$ 
complex-valued matrix  satisfying \eqref{eqDS1.1}
and \eqref{eqDS1.2}, that is,
the sum of the entries at each row is $m$
and at each column  is $n$. Then
\begin{equation}
\label{eqD.1.1}
|\supp A| \geq  n + m - \gcd(n,m).
\end{equation}
\end{thm}

\begin{proof}
Let $G(A) = (U,V,E,w)$ be the graph of
the matrix $A$, and let
 $r$ be the number of 
connected components of $G(A)$.
Since the graph has $n+m$ vertices, the number
of edges in the graph must be at least $n + m - r$ 
(for adding an edge to a graph 
can reduce the number of
 connected components by at most one). So we obtain
\begin{equation}
\label{eqD.1.2}
|\supp A| = |E| \geq  n + m - r.
\end{equation}

Let $H_1, \dots, H_r$ denote the 
 connected components of $G(A)$. 
For each $1 \leq k \leq r$,
assume that 
$H_k$ has $n_k$ vertices in $U$, and
$m_k$ vertices in $V$.  Since the
sum of  the weights of  the edges incident to 
a vertex from $U$ is $m$, and
the sum of  the weights of  the edges incident to 
a vertex from $V$ is $n$, then
the sum of  the weights of all the edges of 
the subgraph
 $H_k$ is equal to $ m n_k $ on one hand, and 
 to $n  m_k$ on the
other hand. Hence
\begin{equation}
\label{eqD.1.7}
\frac{n_k}{m_k} = \frac{n}{m}
\quad (1 \leq k \leq r).
\end{equation}

Let us now write 
 $n = dp$, $m = dq$, where 
$d = \gcd(n,m)$ and 
$p,q$ are coprime. It follows from
\eqref{eqD.1.7} that each one of
$n_1, \dots, n_r$ can be no
 smaller than $p$.
But  $n_1 + \dots + n_r = n$,
so we must have $rp \leq n$, or equivalently,
$r \leq d$.
Together with \eqref{eqD.1.2} this implies
\eqref{eqD.1.1}.
\end{proof}

\subsubsection*{Remark}
We later found that a similar argument 
appeared in \cite[Section 4]{DM62}.

\subsection{Block structure of 
arrays with minimal support size}
Now that we have established the lower bound 
\eqref{eqD.1.1}, our next goal is to look into
the structure of  the
 $n \times m$ doubly stochastic arrays 
for which an \emph{equality} in
\eqref{eqD.1.1} is attained.
 We  keep using the notation and terminology
from the proof of \thmref{thmD1.1}. 
In particular we denote
\begin{equation}
\label{eqD8.15.4}
n = dp, \quad m = dq, \quad 
d = \gcd(n,m), \quad 
\text{$p,q$ are coprime.}
\end{equation}

Suppose that $A$ is an
 $n \times m$ doubly stochastic array 
whose  support is of 
size exactly $n + m - d$.
Then all the inequalities in 
the proof above become
equalities. So the number of
 connected components of
the graph $G(A)$ is exactly $d$,
and each connected component 
has exactly $p$ vertices in $U$ and
$q$ vertices in $V$.  This implies
that by 
 a permutation of rows and columns
we may assume that
$A$ has the structure of a block diagonal matrix 
consisting of $d$ blocks of size $p \times q$ each.
Moreover, for each block
the sum of the entries at each row is $m$ and
at each column  is $n$, so each block
is equal to the scalar $d$ times a certain
$p \times q$
doubly stochastic array. That is to say,
\begin{equation}
\label{eqD8.1.4}
A = d \times
\begin{bmatrix}
A_1   \\
 & A_2   \\
 &  & \ddots   \\
 &  &  & A_d   
\end{bmatrix}
\end{equation}
where each one of $A_1, A_2, \dots, A_d$
 is a $p \times q$ doubly stochastic array.

We also observe that the
support size of each $A_k$ can be
no smaller than $p+q-1$, again by
\thmref{thmD1.1}. So we have
\begin{equation}
\label{eqM1.4.16}
|\supp A| = \sum_{k=1}^{d}
|\supp A_k| \geq d(p+q-1) = n+m-d.
\end{equation}
Hence $|\supp A| = n + m - d$
implies that $|\supp A_k|  = p+q-1$
for each $1 \leq k \leq d$.

It is obvious that these conditions, necessary
for an
 $n \times m$ doubly stochastic array  
to have support of 
size $n + m - d$,
are also sufficient. 
That is, suppose that we are given
 $d$ matrices $A_1, A_2, \dots, A_d$
such that each $A_k$ is a $p \times q$
doubly stochastic array with support
of size $p+q-1$. Then the $n \times m$ matrix
$A$ given by
 \eqref{eqD8.1.4} is 
an $n \times m$ doubly stochastic array with
support of size $n + m - d$.
We have thus proved the following:

\begin{thm}
\label{thmD8.44.1}
For any $n,m$, let the numbers  $d,p,q$ be defined
by  \eqref{eqD8.15.4}. Then
an  $n \times m$ doubly stochastic array 
$A$ has support of size $n + m - d$
if and only if, possibly after
 a permutation of rows and columns,
$A$ has 
 the form
 \eqref{eqD8.1.4} where
each one of 
 $A_1, \dots, A_d$
is a  $p \times q$
doubly stochastic array
with support of size $p+q-1$.
\end{thm}

The main point of this result is that
it  reduces
the analysis of
  $n \times m$ doubly stochastic arrays
with  minimal support size 
to the special case when $n,m$
are \emph{coprime}.

\subsection{Constructing arrays with minimal support size}

Our next goal is to prove that the lower bound
\eqref{eqD.1.1} is sharp, that is, to establish
the existence of 
  $n \times m$ doubly stochastic arrays
for which an equality in
\eqref{eqD.1.1} is attained.

\begin{thm}
\label{thmD3.1}
For any $n,m$ there exists
an $n \times m$   doubly stochastic  array $A$  such that
\begin{equation}
\label{eqD.3.1}
|\supp A| = n + m - \gcd(n,m).
\end{equation}
\end{thm}

In fact we only have to prove this
for coprime $n,m$, thanks to \thmref{thmD8.44.1}.

Below we present two different methods for constructing
$n \times m$   doubly stochastic  
arrays of minimal
support size. We call the  first one  the
``discrete trapezoid method'' 
and the second one the ``Euclidean algorithm method''.

\subsubsection{The discrete trapezoid method}
Let $\Z_{nm}$ be the
additive group of residue classes modulo $nm$,
and let $\chi_k$  denote the indicator
function of the subset
$\{0,1,\dots,k-1\}$ of $\Z_{nm}$.
We consider 
 a function $\tau: \Z_{nm} \to \R$ defined
as the convolution
\begin{equation}
\label{eqM3.4}
\tau(t) = (\chi_n \ast \chi_m)(t) = \sum_{s \in \Z_{nm}} \chi_n(t-s)\chi_m(s),
\quad
t \in \Z_{nm}.
\end{equation}
By a straightforward calculation one
can verify that
\begin{equation}
\label{eqM3.4.5.1a}
\tau(t) = \min \{t+1,  n+m-t-1, n, m \}, 
\quad 0 \leq t \leq n+m-2,
\end{equation}
and
\begin{equation}
\label{eqM3.4.5.1b}
\tau(t) = 0, \quad n+m-1  \leq t \leq nm-1.
\end{equation}
We call $\tau$  the \emph{discrete trapezoid} function
on the group $\Z_{nm}$.

For example, in \figref{fig:trapezoid}
the discrete trapezoid function 
is shown
for $n=5$, $m=7$.


\begin{figure}[ht]
\centering

\begin{tikzpicture}[scale=0.41, style=mystyle,thin]

\draw[myroundrect] (-2,-3) rectangle  (36,8);


\foreach \j in {0,1,2,...,34} {
   \coordinate (x\j) at (\j,0);
  }


\foreach \j [evaluate=\j as \val using {int(min(\j+1, 5, 11-\j))}] 
 in {0,1,2,...,10} {
   \coordinate (y\j) at (\j, \val);
}

\foreach \j in {11,12,13,...,34} {
   \coordinate (y\j) at (\j,  0);
  }


\foreach \j [evaluate=\j as \jp using {int(\j+1)}]   in {0,1,2,...,33} {
   \draw[myedge] (y\j) edge  (y\jp);
}


\foreach \j   in {0,1,2,...,10} {
   \draw[mylightedge] (x\j) edge  (y\j);
}


\foreach \j in {0,1,2,...,34} {
   \draw[xnode] (x\j) node[xnodelabel]{$\scriptscriptstyle \j$} circle (\mylittlecirc); 
}


\foreach \j  [evaluate=\j as \val using {int(min(\j+1, 5, 11-\j))}] 
 in {0,1,2,...,10} {
   \draw[ygnode] (y\j) 
       node[ygnodelabel]{$\scriptstyle \val$} circle (\mylittlecirc); 
}

\foreach \j in {11,12,13,...,34} {
   \draw[ygnode] (y\j) circle (\mylittlecirc); 
 }

\end{tikzpicture}

\caption{The discrete trapezoid function for $n=5$, $m=7$.}
\label{fig:trapezoid}
\end{figure}


Now assume  that $n,m$ are \emph{coprime}.
We denote by $\Z_n$ and $\Z_m$ the
additive groups of residue classes modulo $n$
and $m$ respectively. By the
Chinese remainder theorem
 there 
is a group isomorphism 
$\varphi: \Z_{nm} \to \Z_n \times \Z_m$
given by
$\varphi(t) =  
(t \,\operatorname{mod} \, n, \,
t \,\operatorname{mod} \, m)$.
This isomorphism allows us to lift the
discrete trapezoid  function $\tau$
to a new function 
\begin{equation}
\label{eqM3.2}
A: \Z_n \times \Z_m \to \R
\end{equation}
defined by 
$A(\varphi(t)) = \tau(t)$,
$t \in \Z_{nm}$. We use
\eqref{eqM3.2}
as an alternative way to
represent a real
 $n \times m$ matrix $A$,
in which  the rows of the matrix   are 
indexed by residue classes modulo $n$,
while  the columns 
indexed by 
residue classes modulo $m$.

We claim that the matrix $A$ thus constructed
is  doubly 
stochastic. To see this, 
let $H_k$  denote the
subgroup of $\Z_{nm}$ generated
by the element $k$. We then observe that
 the discrete trapezoid function $\tau$
\emph{tiles} the group $\Z_{nm}$ by translations
along each one of 
the two subgroups $H_n$ and $H_m$. That is to say,
\begin{equation}
\label{eqM1.15}
\sum_{s \in H_n} \tau(t-s) = m, \quad
\sum_{s \in H_m} \tau(t-s) = n, \quad
t \in \Z_{nm}.
\end{equation}
This can be verified directly using the
definition \eqref{eqM3.4} of the function
$\tau$.

We next observe that 
the isomorphism 
$\varphi$ maps the subgroup
$H_n$ of $\Z_{nm}$ onto 
the subgroup
$\{0\} \times \Z_m$
of $\Z_n \times \Z_m$.
Hence for each $i \in \Z_n$, the set
$\{(i,j) : j \in \Z_m\}$ is the image under
$\varphi$ of a certain coset of $H_n$ in $\Z_{nm}$,
say, the coset $a_i - H_n$. It follows that 
\begin{equation}
\label{eqM1.12}
\sum_{j \in \Z_m} A(i,j) = 
\sum_{s \in H_n } \tau(a_i-s) = m,
\end{equation}
where in the last equality we used \eqref{eqM1.15}.
In a similar way, $\varphi$ maps the subgroup
$H_m$ onto  $\Z_n \times \{0\}$,
so for each $j \in \Z_m$ the set
$\{(i,j) : i \in \Z_n\}$ is the image under
$\varphi$ of a coset $b_j - H_m$, and we obtain
\begin{equation}
\label{eqM1.13}
\sum_{i \in \Z_n} A(i,j) = 
\sum_{s \in H_m } \tau(b_j-s) = n,
\end{equation}
again using \eqref{eqM1.15}.
Since the values of $A$ are clearly nonnegative,
we  conclude  from 
\eqref{eqM1.12} and 
\eqref{eqM1.13} that 
$A$ is a doubly 
stochastic array.

 We note that the correspondence between  
the $n \times m$ doubly 
stochastic arrays and the
non\-negative functions which 
tile the group $\Z_{nm}$ 
by translations
along each one of 
the two subgroups $H_n$ and $H_m$
(where $n,m$ are coprime)
was pointed out in an earlier version of
\cite{KP22}.

Finally we notice, using
\eqref{eqM3.4.5.1a} and
\eqref{eqM3.4.5.1b},  that
\begin{equation}
\label{eqM1.17}
|\supp A| = |\supp \tau| = n+m-1,
\end{equation}
so that the support of  $A$ is of
the smallest possible size  
according  to \thmref{thmD1.1}.

For example, in the case
 $n=5$, $m=7$ we obtain the matrix
\begin{equation}
\label{eqM1.9}
\begin{bmatrix}
1 & & & 1 & & 5   \\
 & 2 & & & &  & 5  \\
4 &  & 3   \\
  & 3 &  & 4   \\
  &  & 2 &  & 5 
\end{bmatrix}
\end{equation}
which is a $5 \times 7$ doubly 
stochastic array that has $11 = 5 + 7 - 1$ nonzero
entries.

One other example is 
the matrix on the right hand side of \eqref{eqDS1.3},
which 
is obtained by the discrete trapezoid method
for $n=3$, $m=4$.

\subsubsection{The Euclidean algorithm method}
\label{subsubsec:euclid}
We now turn to describe our second method  for constructing
$n \times m$ doubly stochastic  
arrays of minimal
support size. The same method was 
proposed also in \cite{Lou23}. 

Assume that  $n \leq m$, and write
\begin{equation}
\label{eqM2.11}
m = kn + r, \quad k \geq 1, \quad 0 \leq r \leq  n-1.
\end{equation}

If $r=0$ then $m = kn$, and in this case we already know from
\propref{propPR5.7} how to construct an
  $n \times m$ doubly stochastic array
of the minimal support size $kn$. 

In the remaining case $1 \leq r \leq  n-1$ we proceed by induction as follows. 
We first use  the inductive hypothesis to find  an $r \times n$
doubly stochastic array $B$ of support size 
$r +n - \gcd(r,n)$. Next, let
 $D_n$ denote  the $n \times n$ identity matrix times the scalar $n$,
that is, $D_n$ is  an $n \times n$ diagonal matrix whose
 main diagonal entries are all equal to $n$. 
We then construct an $n \times m$ matrix $A$ given by
\begin{equation}
\label{eqM2.12}
A = \begin{bmatrix}
D_n & D_n & \cdots & D_n & B^{\top}
\end{bmatrix},
\end{equation}
that is, $A$ is
obtained by horizontally concatenating 
$k$ copies of $D_n$ and the transpose of the matrix $B$.
Then $A$ has nonnegative entries and
 the sum of the entries at each row is $m$
 and at each column  is $n$. Hence $A$ is
an  $n \times m$ doubly stochastic array.

It remains to find the size of the support of $A$.
We have
\begin{equation}
\label{eqM2.17}
|\supp A| = k |\supp D_n| + |\supp B|
= kn + r+n-\gcd(r,n).
\end{equation}
But $kn+r = m$ and $\gcd(r,n) = \gcd(n,m)$, so we 
conclude  from
\eqref{eqM2.17} that
\begin{equation}
\label{eqM2.18}
|\supp A| = m + n - \gcd(n,m),
\end{equation}
which is indeed the smallest possible support size.

For example, applying this algorithm for
 $n=5$, $m=7$ yields the matrix
\begin{equation}
\label{eqM1.16}
\begin{bmatrix}
5 & & & & & 2 &  \\
 & 5 & & & &  & 2 \\
 &  & 5 & & & 2 &  \\
  &  &  & 5 & &  & 2 \\
  &  &  &  & 5 & 1 & 1  
\end{bmatrix}
\end{equation}
which is a $5 \times 7$ doubly 
stochastic array with support of size
 $11 = 5 + 7 - 1$.

One more example is 
the matrix on the left hand side of \eqref{eqDS1.3},
which 
is obtained by the  Euclidean algorithm method
for $n=3$, $m=4$.

\subsubsection{Remarks}
In general, the discrete trapezoid method
and the  Euclidean algorithm method
may yield different, and in fact non-equivalent,  doubly 
stochastic arrays. For example, 
the $5 \times 7$ matrices 
in \eqref{eqM1.9}
and  \eqref{eqM1.16}, obtained
by the two different methods,  are obviously
not equivalent to each other.

We also point out that  when
 performing
the Euclidean algorithm method
for $n,m$ coprime, 
we may choose 
at any stage of the induction
to use 
the discrete trapezoid method
rather than continue further with the
inductive process.
In this way one can obtain more examples 
of $n \times m$   doubly stochastic  
arrays of minimal
support size.


\section{Extremal arrays with non-minimal support size}
\label{secD2}

In this section we prove \thmref{thmA5.1}
that characterizes the possible sizes of the supports
of all extremal $n \times m$ arrays. Recall that
the theorem 
asserts that if neither one of $n,m$ divides the other,
then  any extremal $A \in \Snm$
has support size of   the form
\begin{equation}
\label{eqD2.9.1}
|\supp A| = n + m - s, 
\end{equation}
where
\begin{equation}
\label{eqD2.9.2}
1 \leq s \leq \gcd(n,m). 
\end{equation}
Moreover, this condition
is sharp, that is, for any 
integer $s$ satisfying \eqref{eqD2.9.2}
 there exists an  extremal array 
 $A \in \Snm$  whose support size 
is given by \eqref{eqD2.9.1}

(The reader is reminded that 
this is no longer  true if one of $n,m$ 
does divide the other. In fact, if $m=kn$ then
every extremal array in $\S_{n, kn}$ has
support of size  $kn$,
by  \propref{propPR5.7}.)

\subsection{Estimating the support size of an extremal array}
We first show that
the conditions \eqref{eqD2.9.1}
and \eqref{eqD2.9.2} hold for 
any extremal
$n \times m$ doubly stochastic  array.

\begin{thm}
\label{thmE4.2}
For all $n,m$ 
the  support size of   
any extremal array  in $\Snm$ is of
the form
$n + m - s$, where $1 \leq s \leq \gcd(n,m)$. 
\end{thm}

\begin{proof}
By  \thmref{thmD1.1} we  know that 
every $A \in \Snm$ has support of size no less
than $n + m - \gcd(n,m)$. Now let 
 $A$ be extremal, then by \propref{propPR5.4}
the graph $G(A)$  contains no cycles.
Since the graph has $n+m$ vertices, this implies
that the number
of edges in the graph 
must be $n + m - r$, where 
 $r$ is the number of 
connected components of the graph. 
Hence $\supp A$ 
is of size $n + m - r$,
which can be no greater 
than $n + m - 1$.
\end{proof}

\subsection{Constructing extremal arrays of given support size}

Next we turn to  prove that 
 if neither one of $n,m$ divides the other,
then the estimate given in \thmref{thmE4.2} is sharp.
That is, we will construct examples of extremal $n \times m$
  doubly stochastic   arrays  whose
support sizes attain all values of the form $n+m-s$,
where $1 \leq s \leq \gcd(n,m)$.

\begin{thm}
\label{thmE4.9}
Assume that neither one of $n,m$ divides the other.
Then for any $s$ such that $1 \leq s \leq \gcd(n,m)$,
there exists an  extremal array 
 $A \in \Snm$  whose support is of size $n+m-s$.
\end{thm}

Our proof consists of 
two parts. In the first part, which
is the key one in the  proof,
we establish 
the result in the special case where
$\gcd(n,m)= m-n$.
Then in the second part we use the
Euclidean algorithm method
to extend the result to 
the general case.

\subsubsection{}
The key part of  our construction
is performed in the following lemma:

\begin{lem}
\label{lemE4.1}
Let $n = dp$, $m= d(p+1)$,
where $d>1$, $p>1$.
Then for any  $1 \leq s \leq d-1$,
there is an  extremal  array
 $A \in \Snm$  whose support is of size $n+m-s$.
\end{lem}

\begin{proof}
It will be more convenient to construct the 
matrix $B = d^{-1} \times A$ 
(recall that by \propref{propPR5.2}
 all the entries of $A$ must be 
integral multiples of $d$, so that
$B$ would be an integer matrix).
We will obtain
the matrix $B$ by constructing
its associated  graph $G(B) = (U,V,E,w)$.
The graph $G(B)$ will be decomposed from 
a system of $d$ weighted bipartite subgraphs
$G_j = (U_j, V_j, E_j, w_j)$, 
$1 \leq j \leq d$,
together with additional edges that join
vertices between these subgraphs.

The  subgraphs $G_1, \dots, G_d$ will be constructed
from four basic graph prototypes that will be now introduced.
Each prototype is a weighted bipartite graph 
containing no cycles and such that the
edges have positive weights. Furthermore,
the weight of each $U$-vertex is $p+1$ and 
the weight  of each $V$-vertex 
is $p$, with possibly the exception of certain
``special'' vertices whose weight  has deficiency one.
That is, a special $U$-vertex has weight $p$
and a special $V$-vertex has weight $p-1$.

The four graph prototypes are
illustrated in  Figures
\ref{fig:typeiandii} and \ref{fig:typeiiiandiv}.
The first prototype, type I, has $p-1$ vertices in $U$
and $p$ vertices  in $V$, and has one
 special $V$-vertex. The second, 
type II, has $p+1$ vertices in $U$
and $p+2$ vertices  in $V$, and has one
 special $U$-vertex. 
Types III and IV have $p$ vertices in $U$
and $p+1$ vertices  in $V$, but type III
has one  special $U$-vertex and one 
special $V$-vertex, while type IV has
no special vertices.


\begin{figure}[t]
\centering

\begin{subfigure}[t]{0.475\textwidth}
        \centering
\begin{tikzpicture}[scale=0.32, style=mystyle]


\draw[myroundrect] (-7,-4) rectangle  (15,42);

\coordinate (v1) at (8,4);
\coordinate (u1) at (0,6);
\coordinate (v2) at (8,8);
\coordinate (u2) at (0,10);
\coordinate (v3) at (8,12);
\coordinate (u3) at (0,14);
\coordinate (v4) at (8,16);

\coordinate (vx0) at (8,34);
\coordinate (ux1) at (0,32);
\coordinate (vx1) at (8,30);
\coordinate (ux2) at (0,28);
\coordinate (vx2) at (8,26);
\coordinate (ux3) at (0,24);
\coordinate (vx3) at (8,22);
\coordinate (ux4) at (0,20);

\coordinate (dotsu) at (-2,18);
\coordinate (dotsv) at (10,18);

\draw[specialnode] (vx0)  circle (\mybigcirc);

\draw[myedge] (v1) edge node[vedgelabel]{$\myk$}  (u1);
\draw[myedge] (u1) edge node[vedgelabel]{$1$}  (v2);
\draw[myedge] (v2) edge node[vedgelabel]{$\myk-1$}  (u2);
\draw[myedge] (u2) edge node[vedgelabel]{$2$}  (v3);
\draw[myedge] (v3) edge node[vedgelabel]{$\myk-2$}  (u3);

\draw[myedge] (vx0) edge node[uedgelabel]{$\myk-1$}  (ux1);
\draw[myedge] (ux1) edge node[uedgelabel]{$2$}  (vx1);
\draw[myedge] (vx1) edge node[uedgelabel]{$\myk-2$}  (ux2);
\draw[myedge] (ux2) edge node[uedgelabel]{$3$}  (vx2);
\draw[myedge] (vx2) edge node[uedgelabel]{$\myk-3$}  (ux3);
\draw[myedge] (ux3) edge node[uedgelabel]{$4$}  (vx3);

\draw (dotsu) node{$\vdots$} {};
\draw (dotsv) node{$\vdots$} {};

\draw[myedge, dashed] (u3) -- (v4);
\draw[myedge, dashed] (vx3) -- (ux4);

\draw[vnode] (vx0) node[vnodelabel]{$v_\myk$} circle (\mycirc);
\foreach \j in {1,...,3} { \draw[vnode] (vx\j) node[vnodelabel]{$v_{\myk-\j}$} circle (\mycirc); }
\foreach \j in {1,...,3} { \draw[unode] (ux\j) node[unodelabel]{$u_{\myk-\j}$} circle (\mycirc); }
\foreach \j in {1,...,3} { \draw[vnode] (v\j) node[vnodelabel]{$v_\j$} circle (\mycirc); }
\foreach \j in {1,...,3} { \draw[unode] (u\j) node[unodelabel]{$u_\j$} circle (\mycirc); }

\end{tikzpicture}
    \end{subfigure}
~
\begin{subfigure}[t]{0.475\textwidth}
        \centering

\begin{tikzpicture}[scale=0.32, style=mystyle]


\draw[myroundrect] (-7,-4) rectangle  (15,42);

\coordinate (v1) at (8,0);
\coordinate (u1) at (0,2);
\coordinate (v2) at (8,4);
\coordinate (u2) at (0,6);
\coordinate (v3) at (8,8);
\coordinate (u3) at (0,10);
\coordinate (v4) at (8,12);

\coordinate (vxp2) at (8,38);
\coordinate (uxp1) at (0,36);
\coordinate (vxp1) at (8,34);
\coordinate (ux0) at (0,32);
\coordinate (vx0) at (8,30);
\coordinate (ux1) at (0,28);
\coordinate (vx1) at (8,26);
\coordinate (ux2) at (0,24);
\coordinate (vx2) at (8,22);
\coordinate (ux3) at (0,20);
\coordinate (vx3) at (8,18);
\coordinate (ux4) at (0,16);

\coordinate (dotsu) at (-2,14);
\coordinate (dotsv) at (10,14);

 \draw[specialnode] (ux0)  circle (\mybigcirc);

\draw[myedge] (v1) edge node[vedgelabel]{$\myk$}  (u1);
\draw[myedge] (u1) edge node[vedgelabel]{$1$}  (v2);
\draw[myedge] (v2) edge node[vedgelabel]{$\myk-1$}  (u2);
\draw[myedge] (u2) edge node[vedgelabel]{$2$}  (v3);
\draw[myedge] (v3) edge node[vedgelabel]{$\myk-2$}  (u3);

\draw[myedge] (vxp2) edge node[uedgelabel]{$\myk$}  (uxp1);
\draw[myedge] (uxp1) edge node[uedgelabel]{$1$}  (vxp1);
\draw[myedge] (vxp1) edge node[uedgelabel]{$\myk-1$}  (ux0);
\draw[myedge] (ux0) edge node[uedgelabel]{$1$}  (vx0);
\draw[myedge] (vx0) edge node[uedgelabel]{$\myk-1$}  (ux1);
\draw[myedge] (ux1) edge node[uedgelabel]{$2$}  (vx1);
\draw[myedge] (vx1) edge node[uedgelabel]{$\myk-2$}  (ux2);
\draw[myedge] (ux2) edge node[uedgelabel]{$3$}  (vx2);
\draw[myedge] (vx2) edge node[uedgelabel]{$\myk-3$}  (ux3);
\draw[myedge] (ux3) edge node[uedgelabel]{$4$}  (vx3);

\draw (dotsu) node{$\vdots$} {};
\draw (dotsv) node{$\vdots$} {};

\draw[myedge, dashed] (u3) -- (v4);
\draw[myedge, dashed] (vx3) -- (ux4);

\draw[vnode] (vxp2) node[vnodelabel]{$v_{\myk+2}$} circle (\mycirc);
\draw[vnode] (vxp1) node[vnodelabel]{$v_{\myk+1}$} circle (\mycirc);
\draw[vnode] (vx0) node[vnodelabel]{$v_{\myk}$} circle (\mycirc);
\foreach \j in {1,...,3} { \draw[vnode] (vx\j) node[vnodelabel]{$v_{\myk-\j}$} circle (\mycirc); }

 \draw[unode] (uxp1) node[unodelabel]{$u_{\myk+1}$} circle (\mycirc); 
 \draw[unode] (ux0) node[unodelabel]{$u_{\myk}$} circle (\mycirc); 
\foreach \j in {1,...,3} { \draw[unode] (ux\j) node[unodelabel]{$u_{\myk-\j}$} circle (\mycirc); }
\foreach \j in {1,...,3} { \draw[vnode] (v\j) node[vnodelabel]{$v_\j$} circle (\mycirc); }
\foreach \j in {1,...,3} { \draw[unode] (u\j) node[unodelabel]{$u_\j$} circle (\mycirc); }

\end{tikzpicture}
    \end{subfigure}

\caption{Graph prototypes of type I (left) and type II (right).}
\label{fig:typeiandii}
 \end{figure}


\begin{figure}[t]
\centering

\begin{subfigure}[t]{0.475\textwidth}
        \centering
\begin{tikzpicture}[scale=0.32, style=mystyle]


\draw[myroundrect] (-7,-4) rectangle  (15,38);

\coordinate (v1) at (8,0);
\coordinate (u1) at (0,2);
\coordinate (v2) at (8,4);
\coordinate (u2) at (0,6);
\coordinate (v3) at (8,8);
\coordinate (u3) at (0,10);
\coordinate (v4) at (8,12);

\coordinate (vxp1) at (8,34);
\coordinate (ux0) at (0,32);
\coordinate (vx0) at (8,30);
\coordinate (ux1) at (0,28);
\coordinate (vx1) at (8,26);
\coordinate (ux2) at (0,24);
\coordinate (vx2) at (8,22);
\coordinate (ux3) at (0,20);
\coordinate (vx3) at (8,18);
\coordinate (ux4) at (0,16);

\coordinate (dotsu) at (-2,14);
\coordinate (dotsv) at (10,14);

\draw[specialnode] (ux0)  circle (\mybigcirc);
\draw[specialnode] (vxp1)  circle (\mybigcirc);

\draw[myedge] (v1) edge node[vedgelabel]{$\myk$}  (u1);
\draw[myedge] (u1) edge node[vedgelabel]{$1$}  (v2);
\draw[myedge] (v2) edge node[vedgelabel]{$\myk-1$}  (u2);
\draw[myedge] (u2) edge node[vedgelabel]{$2$}  (v3);
\draw[myedge] (v3) edge node[vedgelabel]{$\myk-2$}  (u3);

\draw[myedge] (vxp1) edge node[uedgelabel]{$\myk-1$}  (ux0);
\draw[myedge] (ux0) edge node[uedgelabel]{$1$}  (vx0);
\draw[myedge] (vx0) edge node[uedgelabel]{$\myk-1$}  (ux1);
\draw[myedge] (ux1) edge node[uedgelabel]{$2$}  (vx1);
\draw[myedge] (vx1) edge node[uedgelabel]{$\myk-2$}  (ux2);
\draw[myedge] (ux2) edge node[uedgelabel]{$3$}  (vx2);
\draw[myedge] (vx2) edge node[uedgelabel]{$\myk-3$}  (ux3);
\draw[myedge] (ux3) edge node[uedgelabel]{$4$}  (vx3);

\draw (dotsu) node{$\vdots$} {};
\draw (dotsv) node{$\vdots$} {};

\draw[myedge, dashed] (u3) -- (v4);
\draw[myedge, dashed] (vx3) -- (ux4);

\draw[vnode] (vxp1) node[vnodelabel]{$v_{\myk+1}$} circle (\mycirc);
\draw[vnode] (vx0) node[vnodelabel]{$v_{\myk}$} circle (\mycirc);
\foreach \j in {1,...,3} { \draw[vnode] (vx\j) node[vnodelabel]{$v_{\myk-\j}$} circle (\mycirc); }

 \draw[unode] (ux0) node[unodelabel]{$u_{\myk}$} circle (\mycirc); 
\foreach \j in {1,...,3} { \draw[unode] (ux\j) node[unodelabel]{$u_{\myk-\j}$} circle (\mycirc); }
\foreach \j in {1,...,3} { \draw[vnode] (v\j) node[vnodelabel]{$v_\j$} circle (\mycirc); }
\foreach \j in {1,...,3} { \draw[unode] (u\j) node[unodelabel]{$u_\j$} circle (\mycirc); }

\end{tikzpicture}
    \end{subfigure}
~
\begin{subfigure}[t]{0.475\textwidth}
        \centering

\begin{tikzpicture}[scale=0.32, style=mystyle]


\draw[myroundrect] (-7,-4) rectangle  (15,38);

\coordinate (v1) at (8,0);
\coordinate (u1) at (0,2);
\coordinate (v2) at (8,4);
\coordinate (u2) at (0,6);
\coordinate (v3) at (8,8);
\coordinate (u3) at (0,10);
\coordinate (v4) at (8,12);

\coordinate (vxp1) at (8,34);
\coordinate (ux0) at (0,32);
\coordinate (vx0) at (8,30);
\coordinate (ux1) at (0,28);
\coordinate (vx1) at (8,26);
\coordinate (ux2) at (0,24);
\coordinate (vx2) at (8,22);
\coordinate (ux3) at (0,20);
\coordinate (vx3) at (8,18);
\coordinate (ux4) at (0,16);

\coordinate (dotsu) at (-2,14);
\coordinate (dotsv) at (10,14);

\draw[myedge] (v1) edge node[vedgelabel]{$\myk$}  (u1);
\draw[myedge] (u1) edge node[vedgelabel]{$1$}  (v2);
\draw[myedge] (v2) edge node[vedgelabel]{$\myk-1$}  (u2);
\draw[myedge] (u2) edge node[vedgelabel]{$2$}  (v3);
\draw[myedge] (v3) edge node[vedgelabel]{$\myk-2$}  (u3);

\draw[myedge] (vxp1) edge node[uedgelabel]{$\myk$}  (ux0);
\draw[myedge] (ux0) edge node[uedgelabel]{$1$}  (vx0);
\draw[myedge] (vx0) edge node[uedgelabel]{$\myk-1$}  (ux1);
\draw[myedge] (ux1) edge node[uedgelabel]{$2$}  (vx1);
\draw[myedge] (vx1) edge node[uedgelabel]{$\myk-2$}  (ux2);
\draw[myedge] (ux2) edge node[uedgelabel]{$3$}  (vx2);
\draw[myedge] (vx2) edge node[uedgelabel]{$\myk-3$}  (ux3);
\draw[myedge] (ux3) edge node[uedgelabel]{$4$}  (vx3);

\draw (dotsu) node{$\vdots$} {};
\draw (dotsv) node{$\vdots$} {};

\draw[myedge, dashed] (u3) -- (v4);
\draw[myedge, dashed] (vx3) -- (ux4);

\draw[vnode] (vxp1) node[vnodelabel]{$v_{\myk+1}$} circle (\mycirc);
\draw[vnode] (vx0) node[vnodelabel]{$v_{\myk}$} circle (\mycirc);
\foreach \j in {1,...,3} { \draw[vnode] (vx\j) node[vnodelabel]{$v_{\myk-\j}$} circle (\mycirc); }

 \draw[unode] (ux0) node[unodelabel]{$u_{\myk}$} circle (\mycirc); 
\foreach \j in {1,...,3} { \draw[unode] (ux\j) node[unodelabel]{$u_{\myk-\j}$} circle (\mycirc); }
\foreach \j in {1,...,3} { \draw[vnode] (v\j) node[vnodelabel]{$v_\j$} circle (\mycirc); }
\foreach \j in {1,...,3} { \draw[unode] (u\j) node[unodelabel]{$u_\j$} circle (\mycirc); }

\end{tikzpicture}
    \end{subfigure}

\caption{Graph prototypes of type III (left) and type IV (right).}
\label{fig:typeiiiandiv}
 \end{figure}


We now use these 
four graph prototypes to
construct  the subgraphs $G_1,\dots,G_d$
as follows. We first let $G_1$  be a graph
of type I. Next, for each $2 \leq j \leq d-s$ we let $G_j$
 be a graph of type III (in the case $s=d-1$ there will
be no subgraphs  of type III). Then, 
we let $G_{d-s+1}$  be a graph
of type II. Finally,  for each $d-s+2 \leq j \leq d$ we let $G_j$
be a graph of type IV (if $s=1$, there will
be no subgraphs  of type IV). 

Next, we 
connect the first 
$d-s+1$ subgraphs
$G_1, \dots, G_{d-s+1}$
by adding edges between them
as follows.  For each $1 \leq j \leq d-s$, we 
add an edge of weight one 
joining the special $V$-vertex of $G_j$
to the special $U$-vertex of $G_{j+1}$.
In  this way, the subgraphs
$G_1, \dots, G_{d-s+1}$ together with
the new edges  become a
single connected component.
Note that we do not add any edges 
to the subgraphs
$G_j$ for  $d-s+2 \leq j \leq d$.

This construction yields a weighted
 bipartite graph $G = (U,V,E,w)$ 
with edges of positive weights,
that has  $dp= n$ vertices in $U$
and $d(p+1) = m$ vertices in $V$.
Moreover, the weight of each $U$-vertex 
is $p+1$ and 
the weight  of each $V$-vertex 
is $p$, with no exceptions
(since by adding  the
new edges we have increased  the weight of every 
special vertex  by exactly one, so that in the final
graph the special vertices have no weight deficiency).
We therefore  have $G=G(B)$ for a certain
$n \times m$ matrix $B$, such that 
$A := d \times B$ is an $n \times m$
  doubly stochastic   array.

Finally, we note that by our construction the
graph $G(B)$, and hence also the graph $G(A)$,
contains no cycles. This implies that
 $A$ is extremal by
\propref{propPR5.4}.
Since the graph $G(A)$ has $n+m$ vertices, this 
also implies
that the number
of edges must be $n + m - r$, where 
 $r$ is the number of 
connected components of the graph. 
But $G(A)$ has exactly $s$ connected 
components, so that 
$|\supp A| = |E| = n+m-s$
as required.
\end{proof}

\subsubsection{}
As an example, if $d=3$, $p=2$, $s=1$, then 
the proof of \lemref{lemE4.1} yields the  $6 \times 9$  doubly  stochastic array
\begin{equation}
\label{eqE4.9.1}
3  \times \begin{bmatrix}
2 & 1    \\
 & & 2 & 1    \\
 & 1 &  & 1 & 1  \\
 &  &  &  &  & 2 & 1 \\
 &  &  &  & 1 &  & 1 & 1 \\
 &  &  &  &  &  &  & 1 & 2 \\
\end{bmatrix}
\end{equation}
which is  extremal and has
 support of the maximal size
$6+9-1 = 14$, while we know that the 
 minimal support size of a
$6 \times 9$  doubly stochastic 
array is $6+9-3 = 12$.

\subsubsection{}
Next, we turn to complete  the
proof of  \thmref{thmE4.9}
by extending 
the result to all values of $n,m$
such that neither one divides the other.
The argument uses 
the Euclidean algorithm 
method to reduce the problem to
the case covered  by  \lemref{lemE4.1}.

\begin{proof}[Proof of  \thmref{thmE4.9}]
Let $d = \gcd(n,m)$.
If $s=d$ then the result is a consequence of 
 \thmref{thmD3.1}, so
 we may suppose that $1 \leq s \leq d-1$.
In particular this means that $d > 1$,
that is, $n,m$ are not coprime.
The construction is done by induction.

Assume that  $n \leq m$, and write
\begin{equation}
\label{eqE4.11}
m = kn + r, \quad k \geq 1, \quad 0 \leq r \leq  n-1.
\end{equation}
The remainder $r$ is an integral multiple of $d$.
We observe that $r$ cannot be zero, since
$n$ does not divide $m$ by assumption.
Hence we must have $d \leq r \leq n-1$.

The induction base case is when $k=1$
 and  $r = d$. In this case
we have $n = dp$, $m = d(p+1)$
for some integer $p$ which must be at least two
 (since $n$ does not divide $m$), and
so the result follows from \lemref{lemE4.1}.

If we have either $k \geq 2$,
or $k=1$ and $d+1  \leq r \leq n-1$, 
then we proceed by induction as follows. 
Suppose that we can find an extremal
 $B \in \S_{n,m-n}$ of support size 
$m - s$. Let $D_n$ be   
the $n \times n$ identity matrix times the scalar $n$, and let
$A = \begin{bmatrix}
D_n & B
\end{bmatrix}$
be the $n \times m$ matrix
obtained by horizontally concatenating 
$D_n$ and $B$. Then $A \in \Snm$ and 
\begin{equation}
\label{eqE4.17}
|\supp A| =  |\supp D_n| + |\supp B|
= n + m-s.
\end{equation}

We claim that $A$ is extremal. To see this,
 it would suffice by \propref{propPR5.4}
to show that 
the graph $G(A)$ contains no cycles. Indeed,  
 $G(A)$ is obtained from $G(B)$ by
adding $n$ new vertices  (corresponding to
the columns of $D_n$) 
where each new vertex is joined 
by a single new edge to some vertex of $G(B)$ 
(since each  column of $D_n$
has exactly one nonzero entry). 
Then all the new vertices have
degree one, i.e.\  they are leaves.
Since  $G(B)$ contains no cycles ($B$ being extremal) and
since adding  leaves creates no cycle in the graph,
we see that also $G(A)$ contains no cycles
and  $A$ is  extremal.

It remains to show that indeed there exists an extremal
 $B \in \S_{n,m-n}$ of support size 
$m - s$. This would follow from
 the inductive hypothesis if we can
verify that (i) neither one of $n, m-n$ divides
the other; and (ii)  $\gcd(n,m-n)=d$.
The property (ii) is obvious, so we turn to 
verify  (i). But property (i) is equivalent to 
the assertion
that $\min \{n, m-n\}$ does not coincide with 
$d = \gcd(n,m-n)$. And indeed, if 
 $k \geq 2$ then we have $d < n < m-n$.
If, on the other hand, 
 $k=1$ and $d+1  \leq r \leq n-1$, 
then $m-n$ is equal to $r$,
and we have $d < r  < n$.
Thus in any case both
(i) and (ii) are established
and the existence of $B$ follows from
 the inductive hypothesis.
\end{proof}

\subsection{Remark}
\label{secd5.1.5}
Suppose that
 neither one of $n,m$ divides the other,
and let $A$ be an $n \times m$ 
doubly  stochastic array
with support of size $n+m-s$.
If $s= \gcd(n,m)$ then this implies
that $A$ is extremal,
by \thmref{thmA6.7} and
\propref{propPR5.8}.
On the other hand, one can show that
for any $1 \leq s \leq \gcd(n,m)-1$,
the condition
$|\supp A| =  n+m-s$
does not imply extremality.


\section{Support structure for extremal $n \times (kn+1)$ arrays}
\label{secD4}

In this section we obtain a structural
characterization of the support of an
extremal $n \times m$ doubly stochastic array
in the special case where $m = kn+1$.
That is,  for these values of $n,m$ we characterize the subsets
$\Om \sbt \{1, \dots, n\} \times \{1, \dots, m\}$
such that $\Om = \supp A$ for some 
extremal $n \times m$   array
(\thmref{thmF1.1}).

As an application, we  will
construct a one-to-one correspondence between
the extremal $n \times (n+1)$
doubly stochastic arrays and
the \emph{trees on $n+1$ labeled vertices 
having their edges also labeled}.
This will allow us to  determine 
the total number of extremal arrays in $\S_{n,n+1}$
(\thmref{thmE12.1}), as well as the number of 
\emph{equivalence classes} of these extremal arrays.

\subsection{Proof of \thmref{thmF1.1}}
Our goal now is to prove 
that if   $m = k n+1$, then a subset
 $\Om \sbt \{1, \dots, n\} \times \{1, \dots, m\}$
is the support of an extremal
array $A \in \Snm$ if and only if $\Om$ satisfies
the following two conditions:

\begin{enumerate-roman}
\item
\label{itH1.2.1}
For each $1 \leq i \leq n$, the set 
$\{j : (i,j) \in \Om\}$ contains 
exactly $k+1$ elements;
\item
\label{itH1.2.2}
 $\Om$ does not contain any cycle, that is,
any sequence of the form
\begin{equation}
\label{eqH1.2.2}
(i_1, j_1), \; (i_2, j_1), \; 
(i_2, j_2), \; (i_3, j_2), \; 
\dots, \;
(i_s, j_s), \; (i_1, j_s)
\qquad (s \geq 2)
\end{equation}
where $i_1, \dots, i_s$ are distinct row indices
and $j_1, \dots, j_s$  are distinct column indices.
\end{enumerate-roman}

We remind the reader  that if an extremal $n \times m$ array with a given support
does exist, then it is unique, due to  \propref{propPR5.71}.

\subsubsection{Proof of the necessity of  conditions 
\ref{itF1.2.1} and \ref{itF1.2.2}}
Suppose that $\Om = \supp A$ where
$A \in \Snm$ is extremal. 
We first show that condition
\ref{itF1.2.1} must hold. 
In fact this follows from
 \cite[Proposition 6]{Car96}
but we give a different proof, 
based on  \thmref{thmA5.2}.

Indeed, 
the entries of $A$ are nonnegative 
and the sum of the entries at each column  is
$n$, so there can be no entry greater than $n$.
Since the  sum of the entries at each row is $kn+1$, then
in each row 
there must be at least $k+1$ nonzero entries.
Now if there was a row with
more than $k+1$ nonzero entries, then
 the size of the support of $A$ would be greater than
$n(k+1) = n+m-1$. But this is not possible due to
 \thmref{thmA5.2},
since $n$ and $m$ are coprime. Hence each
row of $A$ contains exactly $k+1$ nonzero entries,
and since $\Om = \supp A$ this establishes
that condition \ref{itF1.2.1} holds.

To see that condition \ref{itF1.2.2} must hold
as well, we observe that 
if $\Om = \supp A$ then this condition
means that  the graph $G(A)$ contains no cycles.
Hence  condition  \ref{itF1.2.2} follows from
\propref{propPR5.4}.
\qed

\subsubsection{Proof of the sufficiency of  conditions 
\ref{itF1.2.1} and \ref{itF1.2.2}}
We now  turn to prove the  converse part of the 
result. That is, we assume that $\Om$ satisfies
the two conditions 
\ref{itF1.2.1} and \ref{itF1.2.2},
and we will show that there is an
extremal $A \in \Snm$
with $\supp A = \Om$.

We will  obtain
the matrix $A$ by constructing
its associated  graph $G(A) = (U,V,E,w)$.
Let 
$U = \{u_1, \dots, u_n\}$ and
$V = \{v_1, \dots, v_m\}$ be two disjoint
vertex sets constituting the
bipartition of the graph. We define
the  set of edges $E$ to include the edge 
 $\{u_i, v_j\}$ 
if and only if $(i,j) \in \Om$. 
Our goal now is to construct 
a weight function $w$ on $E$, which
assigns  a positive weight to each edge, 
in such a way that
 each vertex from $U$ has total weight  $m$, while
each vertex from $V$ has total weight  $n$.

Consider the unweighted graph $G = (U,V,E)$.
By condition \ref{itF1.2.2}, this
 graph  contains no cycles.
Since the graph has $n+m$ vertices, this implies
that the number
of edges in the graph 
must be $n + m - r$, where 
 $r$ is the number of 
connected components of the graph. 
On the other hand,  it follows from
condition  \ref{itF1.2.1}  that the size
of $\Om$, and hence also the number
of edges in the graph, is equal to
$n(k+1) = n+m-1$. 
The number of 
connected components is therefore one, 
which means that the graph $G$ is connected. 
So $G$ is both connected
and contains no cycles, that is, the graph $G$ is a tree.

For each edge $e = \{u,v\}$ of the graph $G$
($u \in U$, $v \in V$)  we let
$G_e = (U, V, E \setminus \{e\})$
be the  subgraph of $G$ 
obtained by the removal of the edge $e$.
Then the graph $G_e$ has exactly two connected
components. We let
 $H_e$ denote the connected component
of  $G_e$ that contains 
 the vertex $u$. The set of vertices of
$H_e$ is then the union of two
disjoint sets $U_e$ and $V_e$,
where 
$U_e \sbt U$ and $V_e \sbt V$.
We now make the following claim:
\begin{claim*}
For every edge $e = \{u,v\} \in E$ $(u \in U, v \in V)$ we have
\begin{equation}
\label{eqF2.2}
|V_e| = k|U_e|.
\end{equation}
\end{claim*}

To see this, observe that by condition  \ref{itF1.2.1}
the degree in $G$ of each vertex
from $U$ is exactly $k+1$.
Hence the degree in $H_e$ of each vertex
from $U_e$ is also $k+1$, with the exception of the
vertex $u$ that has degree only  $k$ (since in the graph $G_e$ the
edge $e$ has been removed). Hence
the total number of edges in $H_e$ is equal to
\begin{equation}
\label{eqF2.3}
\sum_{u' \in U_e} \deg_{H_e}(u') = (|U_e|-1)(k+1) + k =  (k+1) |U_e| - 1,
\end{equation}
where  $\deg_{H_e}(u')$ denotes
the degree of the vertex $u'$ in the graph $H_e$.
On the other hand, observe that
the graph $H_e$ does not contain any  cycles
(since it is a subgraph of $G$), and $H_e$ is connected
by its definition. So the number of edges in $H_e$ is
one less than the number of vertices, that is,  $H_e$ must have
exactly
\begin{equation}
\label{eqF2.4}
|U_e| + |V_e| - 1
\end{equation}
edges. Comparing
\eqref{eqF2.3} and \eqref{eqF2.4} we conclude that
\eqref{eqF2.2} holds, which proves the claim.

We now endow the graph $G$ with a weight function $w$ given by
\begin{equation}
\label{eqF2.8}
w(e) := |U_e|, \quad e \in E.
\end{equation}
It is  obvious that this 
weight function is positive. 
Indeed, if $e = \{u,v\}$ 
($u \in U$, $v \in V$) is an edge in $E$, then 
 the set $U_e$ contains at least the vertex $u$,
hence $U_e$ is nonempty and $w(e)$ is a positive integer.

Let us show that 
 each vertex $v \in V$ has total weight  $n$.
Indeed, since the graph $G$ is a tree
(that is, $G$ is connected and
contains no
cycles) 
then the subgraphs $H_e$,
as $e$ goes through the edges 
incident to $v$, are vertex disjoint
and their union contains all 
the vertices of $G$ except $v$ itself.
Hence  the vertex sets $U_e$, as
$e$ goes through the same edges,
are disjoint
and their union is all of $U$.  
The total weight of the vertex $v$ is therefore
\begin{equation}
\label{eqF2.9}
\sum_{e} w(e) = \sum_{e}  |U_e| = |U| = n,
\end{equation}
where in the sum
$e$ goes through the edges 
incident to $v$.

Next  we show that 
 each vertex $u \in U$ has total weight  $m$.
To see this, we fix the vertex $u$ and
recall that $u$ has degree $k+1$ in the graph $G$.
Since  $G$ is a tree, it follows that
a given vertex $v \in V$ belongs
to $H_e$ for every edge $e$ incident to $u$,
with the exception of the edge $e$ that lies
on the unique path connecting $u$ and $v$.
Hence  each vertex $v \in V$ belongs
to exactly $k$ sets $V_e$
as $e$ goes through the edges 
incident to $u$, which yields
\begin{equation}
\label{eqF2.10}
\sum_{e}   |V_e| =  k |V| = km.
\end{equation}
But together with  \eqref{eqF2.2} this implies that
the total weight of the vertex $u$ is
\begin{equation}
\label{eqF2.18}
\sum_{e} w(e) = \sum_{e}  |U_e| 
= \frac1{k} \sum_{e}   |V_e| = m,
\end{equation}
where in the sum
$e$ goes through the edges 
incident to $u$,
in both \eqref{eqF2.10} and  \eqref{eqF2.18}.

We conclude that the weighted graph
$G = (U,V,E,w)$ thus constructed is the graph
associated to an 
$n \times m$ doubly stochastic   array
$A = (a_{ij})$.
Moreover, $A$ is extremal, as the
graph contains no cycles (\propref{propPR5.4}).
Lastly, we notice that $\supp A = \Om$, since the
 matrix entry $a_{ij}$ is nonzero
if and only if 
 $\{u_i, v_j\}$ is an edge in $E$,
which is the case
if and only if $(i,j) \in \Om$. This
completes the proof of \thmref{thmF1.1}. 
\qed


\subsection{An application to extremal $n \times (n+1)$ arrays}
We can now use \thmref{thmF1.1} in order to
construct  a one-to-one correspondence between
the extremal $n \times (n+1)$
doubly stochastic arrays and
the \emph{trees on $n+1$ labeled vertices 
having their edges also labeled}.
This allows us to find
the exact total number of extremal arrays in $\S_{n,n+1}$
(\thmref{thmE12.1}), as well as to determine the number of 
equivalence classes of these extremal arrays.

\subsubsection{}
We begin by introducing the type
of trees that will be used next.

\begin{definition}
\label{defF1}
Let $T$ be a tree  (that is, a graph which
 is both connected
and contains no cycles) on $n+1$ 
vertices.
 We say that
$T \in \TT_n$ if the
vertices of $T$ are labeled 
as $v_1,\dots,v_{n+1}$, and 
the  edges of $T$ are
also labeled (in some order) as $e_1,\dots,e_n$.
\end{definition}

We think of
the  edges $e_1, \dots ,e_n$ of the tree $T \in \TT_n$
as corresponding to the rows of an $n \times (n+1)$
matrix $A$, 
and of  the vertices $v_1, \dots,v_{n+1}$ of the tree
as  corresponding to
the  columns  of $A$. 
We then have the following result:

\begin{thm}
\label{thmE13}
To each tree $T \in \TT_n$ there corresponds 
a unique  extremal array $A \in \S_{n,n+1}$
 with the following property:
$e_i$ is incident to $v_j$ in $T$
if and only if  $(i,j) \in \supp A$,
 for all $i,j$.
Moreover, this correspondence constitutes  a bijection from 
the set of trees in
$\TT_n$ to the set of extremal arrays in $\S_{n,n+1}$.
\end{thm}

\begin{proof}
Given a tree $T \in \TT_n$, let $B = (b_{ij})$ be its
$n \times (n+1)$  incidence matrix,
defined by $b_{ij}=1$ if 
$e_i$ is incident to $v_j$, 
and $b_{ij}=0$ otherwise.
We then claim that $\Om := \supp B$ 
satisfies the conditions \ref{itF1.2.1} and \ref{itF1.2.2}
(with $k=1$)  in \thmref{thmF1.1}.
Indeed, each edge of $T$ is incident with exactly two vertices, 
hence each row of $B$ contains exactly two 
nonzero entries, and  so
\ref{itF1.2.1} is satisfied. To see that
\ref{itF1.2.2} holds as well, we observe that a cycle in  $\Omega$
corresponds to a cycle in the graph $T$.
 But since $T$ is a tree, it contains no cycles,
and \ref{itF1.2.2} follows. We can therefore use 
 \thmref{thmF1.1} to conclude that
$\Omega$ is the support of some extremal array $A \in \S_{n,n+1}$,
having the property that $e_i$ is incident to $v_j$ in $T$
if and only if  $(i,j) \in \supp A$,
 for all $i,j$.
Moreover, such an $A$ is unique, since
an extremal array is uniquely determined by its support 
(\propref{propPR5.71}).

Conversely, given any extremal 
array $A \in \S_{n,n+1}$, we
 construct an
$n \times (n+1)$  matrix
$B = (b_{ij})$ 
defined by $b_{ij}=1$ if 
$(i,j) \in \supp A$,
and $b_{ij}=0$ otherwise.
The conditions
 \ref{itF1.2.1},
 \ref{itF1.2.2}
 in \thmref{thmF1.1} ensure
that $B$ is the incidence matrix
of a cycle-free graph $T$ with $n+1$
vertices labeled 
as $v_1,\dots,v_{n+1}$ and 
with $n$  edges 
labeled as $e_1,\dots,e_n$.
This means that 
 $T$  is a tree in $\TT_n$
with the property that $e_i$ is incident to $v_j$ in $T$
if and only if  $(i,j) \in \supp A$,
 for all $i,j$. Moreover, 
the incidence matrix $B$
determines the tree $T$  uniquely,
which shows that the
correspondence from 
the set of trees in
$\TT_n$ to the set of extremal arrays in $\S_{n,n+1}$
is indeed a bijection.
\end{proof}

\subsubsection*{Remark}
Let $G(A) = (U,V,E,w)$ be the graph associated to an
extremal $n \times (n+1)$  doubly stochastic array $A$,
where
$U = \{u_1, \dots, u_n\}$ and
$V = \{v_1, \dots, v_{n+1}\}$.
Then each vertex $u_i$ is connected  
in the graph $G(A)$ to exactly two $V$-vertices,
that constitute  the endpoints of the edge $e_i$ in the tree
 $T \in \TT_n$ which   corresponds
 to $A$  through the bijection of 
\thmref{thmE13}.
This means that the (unweighted) graph $G(A)$ is 
obtained from the tree $T$
by inserting a new vertex $u_i$ in the middle of each edge 
$e_i$ of $T$.

\subsubsection{}
We can now use the latter  result in order to prove 
\thmref{thmE12.1},
 which asserts that the 
total number of extremal arrays in $\S_{n,n+1}$
is exactly  $n!  (n+1)^{n-1}$.

\begin{proof}[Proof of \thmref{thmE12.1}]
We invoke the classical Cayley's formula, which states that
 the number of trees with $n+1$ labeled vertices
(but with unlabeled edges)  is $(n+1)^{n-1}$.
Notice that there are $n!$ ways to label
 also the edges of such a tree.
Hence  the number of trees in $\TT_n$ 
is exactly  $n!  (n+1)^{n-1}$. 
But the set of trees in
$\TT_n$ 
 has the same cardinality as 
 the set of extremal arrays in $\S_{n,n+1}$,
due to  \thmref{thmE13}, and the result follows.
\end{proof}

  \thmref{thmE13}
also allows us  to  determine the number of 
extremal  $n \times (n+1)$  
doubly stochastic arrays 
\emph{up to equivalence}.
Recall that two arrays are said to be
equivalent if one can be transformed into the other
by a permutation of rows and columns.
Then two equivalent extremal 
arrays in $\S_{n,n+1}$ 
correspond through the bijection 
of   \thmref{thmE13}
 to two isomorphic labeled trees, 
where two  labeled trees are said to be isomorphic 
if one can be transformed into the other
by a permutation of vertex labels and edge labels.
Hence we arrive at the following conclusion:

\begin{thm}
\label{thmE12.2}
The number of  equivalence classes of 
the extremal arrays in $\S_{n,n+1}$ 
is equal to  the
number of unlabeled trees on $n+1$ vertices.
\end{thm}

Thus, for instance, any extremal $3 \times 4$
 doubly stochastic array is equivalent to one of the 
two matrices in 
\eqref{eqDS1.3}, since
 there are only two unlabeled trees on $4$ vertices.

Unfortunately, 
no closed-form expression is known for the 
number of unlabeled trees on $n+1$ vertices
as a function of $n$.


\section{Non-equivalent extremal arrays with the same entries}
\label{secD3}

The following question was  posed by Loukaki in \cite[Remark 2]{Lou23}:
Do there exist two extremal arrays
 $A,B \in \Snm$, whose supports are both of
the least possible  size $n+m - \gcd(n,m)$, 
and such that $A,B$ have the same set of entries 
counted with multiplicities but $A$ and $B$ are
not equivalent?

In this section 
we prove that the
 question admits an affirmative answer 
in $\S_{n,n+1}$ for each $n \geq 6$,
by constructing 
an example of two extremal arrays $A,B$
with the properties above 
  (\thmref{thmE11.3}).
We also observe that
the same does not hold  for $n \leq 5$.

Our strategy is based on
the correspondence  in
   \thmref{thmE13}
between the extremal arrays in $\S_{n,n+1}$
and the trees on $n+1$
labeled vertices having their edges also labeled.
Recall that 
two equivalent extremal 
arrays in $\S_{n,n+1}$ 
correspond 
 to two isomorphic  trees.
 Hence our goal is to 
 construct two non-isomorphic 
 trees  which induce the same multiset of entries in the 
corresponding  extremal arrays.

\subsection{}
Let $T$ be a tree in $\TT_n$,
that is, $T$ is a tree  on $n+1$ 
vertices labeled 
as $v_1,\dots,v_{n+1}$, and 
the  edges of $T$ are also
 labeled  as $e_1,\dots,e_n$
(recall  \defref{defF1}).
Suppose that $i,j$
are such that the edge $e_i$
is incident to the vertex $v_j$ in the tree $T$.
If we remove the edge $e_i$
from the tree  then we obtain a subgraph 
$G_i(T)$ with exactly two connected components.
We  denote by  $H_{ij}(T)$  the connected component of 
$G_i(T)$ that does \emph{not} contain the vertex $v_j$.

\begin{thm}
\label{thmF1}
Let $T$ be a tree in $\TT_n$,
and let $A = (a_{ij})$ be the
   extremal array in $\S_{n,n+1}$
which   corresponds
 to $T$  through the bijection of 
\thmref{thmE13}.
For each $(i,j) \in \supp A$, the entry
$a_{ij}$  is equal to the number of vertices 
in the connected component $H_{ij}(T)$.
\end{thm}

This result 
 complements  \thmref{thmE13} by
specifying the values of the nonzero entries of the extremal 
$n \times (n+1)$ array $A$
that corresponds to a given tree $T \in \TT_n$.

\begin{proof}[Proof of \thmref{thmF1}]
Let $G(A) = (U,V,E,w)$ be the graph associated to the
extremal
$n \times (n+1)$  array $A$, where
$U = \{u_1, \dots, u_n\}$ and
$V = \{v_1, \dots, v_{n+1}\}$.
By  \thmref{thmE13},
$(i,j) \in \supp A$ if and only if
$e_i$ is incident to $v_j$ in the tree $T$.
Hence  each vertex $u_i$ is connected  
in the graph $G(A)$
  to exactly two $V$-vertices, that constitute 
the endpoints of the edge $e_i$ in  $T$.
This means that the (unweighted) graph $G(A)$ is obtained from the tree $T$
by inserting a new vertex $u_i$ in the middle of each edge 
$e_i$ of $T$.

Now  the matrix entry $a_{ij}$
is  the weight of the edge $\{u_i, v_j\}$
in $G(A)$. We recall from the proof of \thmref{thmF1.1}
(with $k=1$) that
 the weight of the edge $\{u_i, v_j\}$ is
 equal to the number of $U$-vertices in the
 connected component 
containing the vertex $u_i$
in the subgraph of $G(A)$
obtained by the removal of the edge 
$\{u_i, v_j\}$.
But each 
$U$-vertex  corresponds to the edge of 
the tree $T$
 in the middle of which it was inserted,
so the weight of the edge $\{u_i, v_j\}$ is
 also equal to the number of edges 
in the connected component 
$H_{ij}(T)$ 
 together with the edge $e_i$ itself
(which corresponds to the vertex $u_i$).
But this number is the same  as the
number of vertices in $H_{ij}(T)$,
which proves our claim.
\end{proof}

\subsection{}
Consider now the two trees shown in \figref{fig:inductivebase}.
Each tree has  $7$ vertices, but
the trees are not isomorphic, 
since one of them has a vertex of degree $4$ 
while  the other does not. 
We will show that these two trees correspond 
to two non-equivalent
extremal $6 \times 7$ 
doubly stochastic 
arrays $A,B$ that
have the same multiset of entries.

\begin{figure}[t]
\centering

\begin{subfigure}[t]{0.475\textwidth}
        \centering
\begin{tikzpicture}[scale=0.32, style=mystyle]


\draw[myroundrect] (-7,0) rectangle  (15,26);

\coordinate (v1) at (0,4);
\coordinate (v2) at (8,4);
\coordinate (v3) at (4,8);
\coordinate (v4) at (4,12);
\coordinate (v5) at (0,16);
\coordinate (v6) at (8,16);
\coordinate (v7) at (0,20);

\draw[myedge] (v1) edge node[edgelabel, right, pos=0.5]{$e_1$}  (v3);
\draw[myedge] (v2) edge node[edgelabel, left, pos=0.5]{$e_2$}  (v3);
\draw[myedge] (v3) edge node[edgelabel, left, pos=0.5]{$e_3$}  (v4);
\draw[myedge] (v4) edge node[edgelabel, right, pos=0.5]{$e_4$}  (v5);
\draw[myedge] (v4) edge node[edgelabel, left, pos=0.5]{$e_6$}  (v6);
\draw[myedge] (v5) edge node[edgelabel, left, pos=0.5]{$e_5$}  (v7);

\foreach \j in {1,...,7} { \draw[vnode] (v\j) 
circle (\mycirc); }

\end{tikzpicture}
    \end{subfigure}
~
\begin{subfigure}[t]{0.475\textwidth}
        \centering
\begin{tikzpicture}[scale=0.32, style=mystyle]


\draw[myroundrect] (-7,0) rectangle  (15,26);

\coordinate (v1) at (0,4);
\coordinate (v2) at (0,8);
\coordinate (v3) at (0,12);
\coordinate (v4) at (8,12);
\coordinate (v5) at (4,16);
\coordinate (v6) at (0,20);
\coordinate (v7) at (8,20);

\draw[myedge] (v1) edge node[edgelabel,right,pos=0.5]{$\phi(e_1)$}  (v2);
\draw[myedge] (v2) edge node[edgelabel,right,pos=0.5]{$\phi(e_4)$}  (v3);
\draw[myedge] (v3) edge node[edgelabel, left, pos=0.5]{$\phi(e_3)$}  (v5);
\draw[myedge] (v4) edge node[edgelabel, right, pos=0.5]{$\phi(e_2)$}  (v5);
\draw[myedge] (v5) edge node[edgelabel, left, pos=0.5]{$\phi(e_5)$}  (v6);
\draw[myedge] (v5) edge node[edgelabel, right, pos=0.5]{$\phi(e_6)$}  (v7);

\foreach \j in {1,...,7} { \draw[vnode] (v\j) 
circle (\mycirc); }

\end{tikzpicture}
    \end{subfigure}

\caption{An example of two trees on $7$ vertices which correspond to two non-equivalent  extremal $6 \times 7$ arrays with the same multiset of entries.}
\label{fig:inductivebase}
 \end{figure}

Let $A$ be the 
extremal array corresponding to the left tree in 
 \figref{fig:inductivebase}. Then by
 \thmref{thmF1.1},
each row of $A$ contains exactly two nonzero entries,
and by \thmref{thmF1}
the values of the two nonzero entries at the $i$'th row of $A$
are equal to the number of vertices 
in each one of the two connected components
created by the removal of the edge $e_i$ in the left tree.

For example, removing the edge $e_1$
in the left tree creates two connected components,
 one containing $6$ vertices
and the other 
consisting of a single vertex.
Hence the two nonzero values 
in the first row of $A$ are $6$ and $1$.

Now consider the 
extremal array $B$ corresponding to the 
right  tree in 
 \figref{fig:inductivebase}. 
Notice that we have labeled 
the edges of this tree as
$\phi(e_1), \dots, \phi(e_6)$, namely,
 there is a bijection $\phi$ between the edges of the left 
tree and the edges of the right tree.
It is now straightforward to verify that 
for each $1 \leq i \leq 6$, 
the removal of the edge $e_i$
 in the left tree, or 
the removal of the edge $\phi(e_i)$
in the right tree,
 creates two connected components whose
respective number of vertices
is the same for the two trees.
This implies that  the $i$'th
row of $A$ contains the same two nonzero values as the $i$'th
row of $B$, although these two nonzero values 
may appear on different columns of $A$ and $B$.

For example, removing the edge $\phi(e_1)$
in the right tree creates a connected component
with $6$ vertices and another component
 with a single vertex.
Hence the two nonzero values
in the first row of $B$ are $6$ and $1$,
the same two values as in the first row of $A$.

The    two extremal $6 \times 7$ doubly stochastic 
 arrays obtained from this construction are 
\begin{equation}
\label{eqBase6.1}
\begin{bmatrix}
6 & 1 &  &  &  &  & \\
 & 1 & 6 &  &  &  & \\
 & 4 &  & 3 &  &  & \\
 &  &  & 2 & 5 &  & \\
 &  &  &  & 1 & 6 & \\
 &  &  & 1 &  &  & 6
\end{bmatrix}
\quad \text{and} \quad 
\begin{bmatrix}
 6 & 1 &   &  &  &  &  \\
 &  & 1 & 6 &  &   &  \\
 &  & 3 & & 4 &   & \\
 & 5 &  &  & 2 &  & \\
 &  & 1 &  &  &  6 &  \\
 &  & 1 &  &  &   & 6
\end{bmatrix},
\end{equation}
which indeed 
have the same two nonzero
values in each row, and as a consequence,
have the same multiset of entries.
On the other hand  these two arrays 
are not equivalent, 
as one of them has a column with $4$ nonzero entries, while
the other does not
(which, in turn, is due to
one of the trees having a vertex of degree $4$ while
the other tree not having such a vertex).

We have thus proved the case $n=6$ in \thmref{thmE11.3}.

\subsection{}
We now continue with the
proof of  \thmref{thmE11.3} for all $n \geq 6$.
It would suffice for us to find
two non-isomorphic trees  $T=(V, E)$ and
$ T'=(V', E')$ with $n+1$ vertices each,
and an edge bijection $\phi: E \to E'$,
such that the removal of the edge $e_i$
 in the tree $T$, or 
the removal of the edge $\phi(e_i)$
in the tree $T'$,
 creates two connected components whose
respective number of vertices
is the same for the two trees.

The construction is done by induction
as illustrated in \figref{fig:inductivestep}.
Assume that the two trees on $n$ vertices have
already been constructed.
We then insert a new edge $e_{n}$ into
 the left tree, between the edge $e_{n-1}$ and the
vertex common to $e_3$ and $e_4$.
Similarly, we insert a new edge 
$\phi(e_n)$  into the right tree,
between the edge $\phi(e_{n-1})$ and the
vertex common to $\phi(e_2)$, $\phi(e_3)$ and
$\phi(e_5)$.
Thus we obtain two trees on $n+1$ vertices 
and also extend the edge bijection $\phi$ from the previous
step to  these new trees.

\begin{figure}[t]
\centering

\begin{subfigure}[t]{0.475\textwidth}
        \centering
\begin{tikzpicture}[scale=0.32, style=mystyle]


\draw[myroundrect] (-7,0) rectangle  (15,26);

\coordinate (v1) at (-3,4);
\coordinate (v2) at (5,4);
\coordinate (v3) at (1,8);
\coordinate (v4) at (1,12);
\coordinate (v5) at (-3,16);
\coordinate (v6) at (3.5,14.5);
\coordinate (v7) at (-3,20);
\coordinate (v8) at (6,17);
\coordinate (v9) at (8.6, 19.6);
\coordinate (v10) at (11.1, 22.1);
\coordinate (dot1) at (6.8, 17.8);
\coordinate (dot2) at (7.3, 18.3);
\coordinate (dot3) at (7.8, 18.8);

\draw[specialnode] (v4)  circle (\mybigcirc);

\draw[myedge] (v1) edge node[edgelabel, right, pos=0.4]{$e_1$}  (v3);
\draw[myedge] (v2) edge node[edgelabel, left, pos=0.4]{$e_2$}  (v3);
\draw[myedge] (v3) edge node[edgelabel, left, pos=0.5]{$e_3$}  (v4);
\draw[myedge] (v4) edge node[edgelabel, right, pos=0.6]{$e_4$}  (v5);
\draw[myedge, dashed] (v4) edge node[edgelabel, right, pos=0.4]{$e_{n}$}  (v6);
\draw[myedge] (v5) edge node[edgelabel, left, pos=0.5]{$e_5$}  (v7);
\draw[myedge] (v6) edge node[edgelabel, right, pos=0.4]{$e_{n-1}$}  (v8);
\draw[myedge] (v9) edge node[edgelabel, right, pos=0.4]{$e_6$}  (v10);

\foreach \j in {1,...,10} { \draw[vnode] (v\j) 
circle (\mycirc); }

\foreach \j in {1,...,3} {\draw (dot\j) 
node{$\cdot$} {};};

\end{tikzpicture}
    \end{subfigure}
~
\begin{subfigure}[t]{0.475\textwidth}
        \centering
\begin{tikzpicture}[scale=0.32, style=mystyle]


\draw[myroundrect] (-7,0) rectangle  (15,26);

\coordinate (v1) at (-3,3);
\coordinate (v2) at (-3,6);
\coordinate (v3) at (-3,9);
\coordinate (v4) at (3,9);
\coordinate (v5) at (0,12);
\coordinate (v6) at (-3,15);
\coordinate (v7) at (2.5,14.5);
\coordinate (v8) at (5, 17);
\coordinate (v9) at (7.8, 19.6);
\coordinate (v10) at (10.3, 22.1);

\coordinate (dot1) at (5.8, 17.8);
\coordinate (dot2) at (6.3, 18.3);
\coordinate (dot3) at (6.8, 18.8);

\draw[specialnode] (v5)  circle (\mybigcirc);

\draw[myedge] (v1) edge node[edgelabel,right,pos=0.5]{$\phi(e_1)$}  (v2);
\draw[myedge] (v2) edge node[edgelabel,right,pos=0.5]{$\phi(e_4)$}  (v3);
\draw[myedge] (v3) edge node[edgelabel, left, pos=0.6]{$\phi(e_3)$}  (v5);
\draw[myedge] (v4) edge node[edgelabel, right, pos=0.6]{$\phi(e_2)$}  (v5);
\draw[myedge] (v5) edge node[edgelabel, left, pos=0.4]{$\phi(e_5)$}  (v6);
\draw[myedge, dashed] (v5) edge node[edgelabel, right, pos=0.4]{$\phi(e_{n})$}  (v7);
\draw[myedge] (v7) edge node[edgelabel, right, pos=0.4]{$\phi(e_{n-1})$}  (v8);
\draw[myedge] (v9) edge node[edgelabel, right, pos=0.4]{$\phi(e_6)$}  (v10);

\foreach \j in {1,...,10} { \draw[vnode] (v\j) 
circle (\mycirc); }

\foreach \j in {1,...,3} {\draw (dot\j) 
node{$\cdot$} {};};

\end{tikzpicture}
    \end{subfigure}

\caption{Constructing the two trees
in the proof of \thmref{thmE11.3}.}
\label{fig:inductivestep}
 \end{figure}

We must show that the new bijection
$\phi$ has the required property.
Indeed, observe that for each $1 \leq i \leq n$,
 the removal of the edge $e_i$ in the left tree, as well as 
 the removal of the 
edge $\phi(e_i)$ in the right tree, creates
one connected component whose number of vertices
is a certain constant $c_i$ which does not depend on $n$,
and a second connected component which
contains the remaining $n+1-c_i$ 
vertices of the tree.

We conclude that  the  two trees thus constructed correspond 
to two extremal $n \times (n+1)$ 
doubly stochastic 
arrays that
have the same two nonzero
values in each row, and hence also
the same multiset of entries.
It remains only to notice that 
the two trees are not isomorphic
(since, as before, 
one of the trees  has a vertex of degree $4$ 
while  the other does not)
and therefore 
the two corresponding extremal arrays are
not equivalent.
The theorem is therefore proved. 
\qed

\subsection{Remark}
We note that the result in 
\thmref{thmE11.3}
 does not hold  for $n \leq 5$.
Indeed, there are
only $13$
unlabeled trees on $n+1$ vertices 
with $1 \leq n \leq 5$ in total
(see, for instance, \cite[Appendix III]{Har69}),
and one can check that no
two of them correspond to
 two extremal doubly stochastic 
arrays that have the same multiset of entries.


\section{Remarks} 
\label{secR1}

We conclude the paper with
some motivational remarks, 
relating the notion of
doubly stochastic arrays
to some other mathematical topics.

\subsection{}
The notion of an $n \times m$ doubly stochastic array admits an obvious
probabilistic interpretation. If $A \in \Snm$ then
the matrix $\frac1{nm} A$ represents a joint
distribution 
of two random variables $X,Y$ 
such that $X$ is uniformly distributed in $\{1,2,\dots,n\}$
and  $Y$ uniformly distributed in $\{1,2,\dots,m\}$.
Thus minimizing the support size of an
array  $A \in \Snm$  is the same as
minimizing the support size of a joint distribution
over all couplings of the two random variables $X,Y$.

\subsection{}
It is also possible to give a combinatorial interpretation
of the  $n \times m$ doubly stochastic arrays as \emph{transportation
plans} as follows. Suppose that we are given $n$
blue bins and $m$ green bins, such that
 each blue bin contains $m$ balls
while the green bins are empty from balls.
The problem is to move the balls from the
blue bins to the green ones, so that each
green bin would contain exactly $n$ balls.
Then a  transportation plan corresponds to
an  integer-valued array $A = (a_{ij})$ in $\Snm$ 
where $a_{ij}$ is the number 
of balls to be moved from the $i$'th blue bin to the
$j$'th green bin. If we seek to minimize
the number of operations 
of the form
``move $a_{ij}$ balls from 
the $i$'th blue bin to the
$j$'th green bin'',
then the minimal number of operations
needed to move the balls from the
blue bins to the green ones as required
is  $n + m - \gcd(n,m)$.
This follows from
 \thmref{thmA6.7}
and the fact that  an array 
in $\Snm$ which
minimizes the
 support size is automatically integer-valued 
(Propositions \ref{propPR5.8}
and \ref{propPR5.2}).

\subsection{}
Let $G$ be a finite abelian group, and suppose that $G$ is
the direct sum of two subgroups $H_1$ and $H_2$
of sizes $n$ and $m$ respectively. A function 
$f$ on the group $G$ is said to
\emph{tile}  by translations
along each one of 
the two subgroups $H_1$ and $H_2$ if we have
\begin{equation}
\label{eqR1.15}
\sum_{s \in H_1} f(t-s) = n, \quad
\sum_{s \in H_2} f(t-s) = m, \quad
t \in G.
\end{equation}

In \cite[Section 4]{KP22} the authors pointed out
a correspondence between 
the $n \times m$ doubly stochastic arrays  and the
\emph{nonnegative} functions $f$  which tile  by translations
along each one of 
the two subgroups $H_1$ and $H_2$.
They posed the question as to
what is the smallest possible size of the support of 
such a function $f$. It follows from
\thmref{thmA6.7}
that the answer  is 
 $n + m - \gcd(n,m)$.
In fact, \thmref{thmD1.1}
implies that the answer remains the same even
for complex-valued functions $f$ 
(that is, even if we do not require $f$ to be nonnegative).


\end{document}